\documentclass[12pt]{article}
\usepackage[top=2.5cm,bottom=2.5cm,left=2.5cm,right=2.5cm]{geometry}
\usepackage{hyperref}   
\usepackage{mathtools}
\usepackage{amssymb}
\usepackage{amsmath,amsthm}
\usepackage[latin1]{inputenc}
\usepackage[dvips]{graphicx}
\usepackage{hyperref}
\usepackage{color}
\usepackage{xcolor}
\usepackage{cite}
\usepackage{mathrsfs}
\usepackage{enumerate}
\usepackage{tikz}
\usepackage{pgffor} 
\usepackage{amsmath, amssymb, mathrsfs}
\usepackage{pgfplots} 
\usetikzlibrary{calc}

\newcommand{\Mod}[1]{\ (\mathrm{mod}\ #1)}
\hypersetup{colorlinks=true, linkcolor=blue, citecolor=blue,urlcolor=blue}

\newtheorem{remark}{Remark}[section]
\newtheorem{definition}[remark]{Definition}

\newtheorem{lemma}[remark]{Lemma}
\newtheorem{theorem}[remark]{Theorem}
\newtheorem{proposition}[remark]{Proposition}

\newtheorem{corollary}[remark]{Corollary}

\usepackage{tikz} 
\pgfplotsset{compat=1.18} 
\usetikzlibrary{backgrounds,positioning}
\usepackage{subcaption} 
\usepackage{pgfmath}
\usepackage{xifthen}
\usetikzlibrary{positioning}

\DeclareMathOperator{\rad}{r}
\DeclareMathOperator{\n}{n}

\DeclareMathOperator{\D}{D}

\newenvironment{acknowledgement}
{\par\bigskip\noindent\textbf{Acknowledgements}\ }
{\par\bigskip}

\title{The Equidistant Dimension of Corona Product Graphs}

\author{
\begin{tabular}{c}
Sandor E. Tu\~{n}\'{o}n-Andr\'{e}s, Alejandro Estrada-Moreno, Juan A. Rodr\'{i}guez-Vel\'{a}zquez \\[1ex]
{\small Universitat Rovira i Virgili, Departament d'Enginyeria Inform\`atica i Matem\`atiques} \\ 
{\small Av. Pa\"{\i}sos Catalans 26, 43007 Tarragona, Spain} \\
{\small \texttt{sandorernesto.tunon@urv.cat, alejandro.estrada@urv.cat, juanalberto.rodriguez@urv.cat}}
\end{tabular}
}

\date{ }

\begin{document}
\maketitle

\begin{abstract}
A subset $S$ of vertices, in a connected graph $G$, is called a distance-equalizer set if for every pair of distinct vertices outside $S$, there exists a vertex in $S$ equidistant to both. The equidistant dimension, denoted by $\xi(G)$, is defined as the minimum cardinality of such sets. While several distance-based parameters have been studied for different graph products, the equidistant dimension of corona product graphs has remained unexplored. In this paper, we investigate the equidistant dimension of the corona product $G \odot H$ of two graphs $G$ and $H$. We introduce the empty bisector graph $\widehat{G}$, an auxiliary construction that relates pairs of vertices in $G$ that cannot be equidistant from any third vertex. Using this framework, we establish tight bounds on the equidistant dimension of $G \odot H$ and derive exact values for several classical families of graphs. Moreover, we show that for any fixed base graph $G$, the equidistant dimension of $G \odot H$ depends on $H$ only through its order and eventually becomes linear in $\n(H)$.

\end{abstract}

{\it Keywords}:
Corona product graph; Distance-equalizer set; Equidistant dimension; Empty bisector graph, Vertex cover.

\section{Introduction}

Distance-based parameters in graph theory have proven to be useful in the analysis of structural, combinatorial, and algorithmic properties of discrete networks. Among these, the \emph{metric dimension} in graphs, introduced by Slater \cite{Slater1975}, and independently by Harary and Melter \cite{Harary1976}, has received extensive attention due to its applications in robot navigation, network verification, and fault detection \cite{Slater1975,Harary1976,Khuller1996,Beerliova2006}. The metric dimension, encapsulates the minimal amount of information required for unique vertex identification based on distances.

More recently, the focus has expanded toward the dual framework of \emph{anonymization} and \emph{privacy} in networks. As a counterpart to metric dimension, Trujillo-Rasua and Yero~\cite{Trujillo-Rasua2016} introduced the concepts of antiresolving sets and metric antidimension, which seek to obfuscate vertex identities rather than reveal them.

Motivated by this line of research, Gonz{\'a}lez, Hernando, and Mora~\cite{Gonzalez2022} proposed a new and more flexible metric parameter, namely the \emph{equidistant dimension} of a graph, which captures the notion of distance-equalizer sets. A subset $S$ of vertices in a connected graph $G$ is called a distance-equalizer set if for every pair of distinct vertices outside $S$, there exists a vertex in $S$ equidistant to both. The equidistant dimension $\xi(G)$ is defined as the minimum cardinality of such sets. This parameter offers a complementary perspective by emphasizing the existence of vertices that are equidistant from a common vertex, which may be of interest in contexts related to symmetry and structural analysis of graphs. So far, connections between distance-equalizer sets and doubly resolving sets have provided new insights and applications \cite{Gonzalez2022}.

Since its inception, the equidistant dimension has been studied from structural and computational viewpoints. Early contributions established general bounds on $\xi(G)$ in terms of classical graph invariants, including the order, diameter, independence number, and maximum degree \cite{Gonzalez2022}. Exact values were subsequently determined for several fundamental graph families, such as paths, cycles, and complete multipartite graphs. More recently, the parameter has been computed for additional structured classes of graphs, notably Johnson and Kneser graphs \cite{Kratica2024}, as well as for various families derived from convex polytopes \cite{Savic2024}. From the algorithmic standpoint, in \cite{GispertFernandez2024} proved that determining the equidistant dimension of a graph is NP-complete, thereby revealing the intrinsic computational complexity of the of the problem.

Graph products constitute a versatile and widely studied operation in graph theory, enabling the construction of complex networks from simpler building blocks. Understanding how graph invariants behave under these products is a classical and challenging problem \cite{Imrich2000}. The equidistant dimension has recently been investigated for lexicographic product graphs \cite{GispertFernandez2024}, and Cartesian product graphs \cite{GispertFernandez2025, Kratica2025}, revealing intricate relationships between the equidistant dimensions of the factors and the product.

In particular, the corona product of graphs offers a useful construction that facilitates the analysis of distance-based and domination parameters,  \cite{Yero2011,Haynes1998a,BarikPatiSarma2007}. 
However, the equidistant dimension of the the corona product graphs remains unexplored. This paper addresses this gap by presenting a comprehensive study of the equidistant dimension of corona product graphs. Given a graph $G$, we introduce the \emph{empty bisector graph} $\widehat{G}$, an auxiliary graph defined on the same vertex set as $G$, where two vertices are adjacent if and only if no vertex of $G$ lies at equal distance from them. This construction captures pairs of vertices for which no third vertex is equidistant, providing a compact encoding of constraints related to the equidistant-dimension. Consequently, the empty bisector graph plays a central role in our analysis.

Using this framework, we derive tight bounds on the equidistant dimension of general corona product graphs in terms of structural parameters of $\widehat{G}$ and the order of the graph $H$. Several illustrative examples are provided, which serve both to clarify the obtained results and to confirm the tightness of the bounds established. Moreover, we obtain closed formulas for the equidistant dimension when the base graph $G$ satisfies certain conditions. These results allow us to determine the exact value of $\xi(G \odot H)$ when $G$ belongs to some classical families of graphs, including paths, cycles, complete graphs, and hypercubes. Furthermore, we prove that $\xi(G \odot H)$ depends on the graph $H$ only through its order $\n(H)$. Combined with structural properties of distance-equalizer sets, this leads to the conclusion that, for any fixed base graph $G$, the equidistant dimension of $G \odot H$ is eventually linear in $\n(H)$. Finally, we establish several relationships between parameters of $G$ and $\widehat{G}$, and we resolve the particular case in which $G$ contains a universal vertex.

The paper is organized as follows. In Section~\ref{Sec:Notation}, we introduce the necessary definitions and preliminaries. 
Section~\ref{Sec:MainResults} investigates fundamental properties of distance-equalizer sets in the corona product graphs, as well as the relationship between the equidistant dimension of the corona product graphs and certain parameters of the empty bisector graph, which serve as a critical foundation for the analyses developed in this paper. Based on these relationships, we obtain sharp bounds for the equidistant dimension of corona product graphs.
In the aforementioned section, we also derive closed formulas for $\xi(G \odot H)$ when $G$ satisfies certain conditions, and we provide a characterization of the extremal case $\xi(G \odot H) = \n(G)$. 
Subsection~\ref{Sec:Computation} applies the previous results to compute exact values of $\xi(G \odot H)$ for several families of connected graphs $G$ and for any graph $H$. 
Section~\ref{Sec:linnh} is devoted to proving that the equidistant dimension becomes linear in $\n(H)$ beyond a certain threshold.
Next, in Section~\ref{Sec:MinorResults}, we explore relationships between parameters of a graph $G$ and its empty bisector graph, and we establish a tight upper bound relating $\xi(G \odot H)$ and $\xi(G)$. We further determine the value of $\xi(G \odot H)$ in the case where $\Delta(G) = \n(G)-1$. Finally,  Section~\ref{Sec:Conclusions} is devoted to conclusions and future works.

\section{Definitions, notation and useful properties} \label{Sec:Notation}

For any graph $G$, the vertex set is denoted by $V(G)$, and the edge set is denoted by $E(G)$. The order of a graph $G = \big(V(G), E(G) \big)$  is denoted by $\n(G)$.
The \emph{degree} of a vertex $v$ in $G$, is denoted by $\delta(v)$, and the \emph{open neighbourhood} of $v$ by $N_G(v)$. 
If $\delta(v) = 1$, then we say that $v$ is an \emph{end vertex}, in which case the only vertex adjacent to $v$ is called its \emph{support vertex}. When $\delta(v) = \n(G) - 1$, we say that $v$ is \emph{universal}. 
The \emph{maximum degree} of $G$ is $\Delta(G) = \max \{\delta(v) : v \in V(G)\}$, and its \emph{minimum degree} is $ \delta(G) = \min \{\delta(v) : v \in V(G)\}.$
The \emph{distance} between two vertices $v,w \in V(G)$ is denoted by $d_G(v,w)$, and the \emph{diameter} of $G$ is 
$\D(G) = \max \{ d_G(v,w) : v,w \in V(G)\}$. The \emph{eccentricity} of a vertex $v \in V(G)$ is $\varepsilon(v) = \max\{ d_G(u,v): u \in V(G)\}$ and the \emph{radius} of $G$ is defined as $\rad(G) = \max\{ \varepsilon(v): v \in V(G)\}$.
A \emph{clique} of a graph $G$ is a complete subgraph of $G$, and the \emph{clique number} of $G$, denoted by $\omega(G)$, is the maximum order of a clique of $G$. 

An \emph{independent set} of $G$ is a subset of pairwise non-adjacent vertices, and the \emph{independence number} of $G$, denoted by $\alpha(G)$, is the maximum cardinality among all independent sets of $G$. 
A \emph{vertex cover} of $G$ is a subset $B \subseteq V(G)$ such that for every edge $\{u,v\} \in E(G)$ we have $\{u,v\}\cap B\ne \varnothing $. The minimum cardinality among all \emph{vertex covers} of $G$ is the \emph{vertex cover number} of $G$ and is usually denoted by $\beta(G)$.

For every integer $n \geq 1$, we denote $[n] = \{1,2,\ldots,n\}$. As usual, we denote by $N_n$, $P_n$, $C_n$, $W_n$, $K_n$, and $K_{r,n-r}$ the empty graphs, the path graphs, the cycle graphs, the wheel graphs, the complete graphs, and the complete bipartite graphs of order $n$, respectively.

All graphs considered in this paper are undirected, simple, and finite. 

\begin{definition}[Distance-equalizer set \cite{Gonzalez2022}]
Let $G$ be a connected graph. A subset $ S \subseteq V(G) $ is called a \emph{distance-equalizer set} if for every pair of distinct vertices $u,v \in V(G) \setminus S $, there exists a vertex $w \in S $ such that
\[
d_G(w, u) = d_G(w, v).
\]
The \emph{equidistant dimension} of $ G $, denoted $\xi(G)$, is the minimum cardinality among all distance-equalizer sets in $G$.
\end{definition}

\begin{definition} [Total distance-equalizer set \cite{GispertFernandez2024}] \label{def:toteqdim}
Let $G$ be a connected graph. A subset $ T \subseteq V(G) $ is a \emph{total distance-equalizer set}, if for every two distinct vertices $u,v \in V(G)$, there exists $w \in T$ such that 
\[
d_G(w, u) = d_G(w, v).
\]
The \emph{Total equidistant dimension} of $G$, denoted $\xi_t(G)$, is the minimum cardinality among all total distance-equalizer sets in $G$.
\end{definition}

Roughly speaking, a bisector is a mathematical object that divides a mathematical structure into two equal parts. In particular, in the context of graph theory, we can state the following definition.

\begin{definition}[Bisector of two vertices]
Given any graph $G$. For every pair of distinct vertices $u, v \in V(G)$, we define the\emph{ bisector} of $u$ and $v$ as
$$
B_G(u \mid v) = \{w \in V(G) : d_G(w, u) = d_G(w, v)\}.
$$
\end{definition}

In this paper, we introduce the following auxiliary graph, which plays a central role in our analysis.

\begin{definition}\emph{(Empty bisector graph).} \label{G_hat}
Given any graph $G = (V, E)$, we define its \emph{empty bisector graph} $\widehat{G} = (V, \widehat{E})$, where $ \{u, v\} \in \widehat{E}$ if and only if  $B_G(u \mid v)= \varnothing$.
\end{definition}

For every connected graph $G$ of order $n$, it is feasible to obtain the graph $\widehat{G}$ in polynomial time of order $O(n^3)$,  using the Floyd-Warshall algorithm \cite{Warshall1962}.

\begin{definition}[Cartesian product of graphs \cite{Imrich2000}] \label{CartesianProd}

Given two graphs $ G = (V_1, E_1) $, $ H = (V_2, E_2) $, the \emph{Cartesian product graph} $ G \, \square \, H = (V_1 \times V_2, E) $ is defined so that two vertices $ (g, h) $ and $ (g', h') $ are adjacent if and only if they satisfy any of the following conditions:

\begin{enumerate}
    \item \( g = g' \) and \( h \sim h' \), or
    \item \( g \sim g' \) and \( h = h' \)
\end{enumerate}

\end{definition}

\begin{definition}[Corona product of graphs\cite{Harary1969}]
Let $G$ and $H$ be two graphs.  
The \emph{corona product} of $G$ and $H$, denoted by $G \odot H=(V,E)$, 
is the graph defined by
$$
V = V(G)\;\cup\; \bigcup_{v \in V(G)} V(H_v),
$$
and
$$
E = E(G)\;\cup\; \bigcup_{v \in V(G)} E(H_v)
    \;\cup\; \bigcup_{v \in V(G)} \{\, \{v,h \} : h \in V(H_v) \,\},
$$
where, for each vertex $v \in V(G)$, the symbol $H_v$ denotes a graph isomorphic to $H$, with vertex and edge sets $V(H_v)$ and $E(H_v)$, respectively.
\end{definition}

From now on, we take $V(G) =\{v_1,\ldots,v_{n}\}$ and $H_i = (V_i,E_i)$ as the $i$-th copy of $H$ in $G\odot H$.  Note that for each $i\in [n]$, all the vertices of $H_i$ are adjacent to $v_i$.

Observe that the corona product $G \odot H$ is connected if and only if $G$ is connected. For any two graphs $G$ and $H$, with $G$ connected, the distances between two vertices $u, v$ in the corona product $G \odot H$, can be computed in terms of the distances in $G$ or in $H$, as follows:

\begin{equation} \label{EqDistGoH}
d_{G \odot H}(u,w) = 
\begin{cases} 
\min\{d_H(u,w), 2\} & \text{if } u,w \in V(H_i)
\\ d_G(v_i,v_j)+2 & \text{if } u \in V(H_i) \text{ and } w \in V(H_j), i\neq j
\\ d_G(v_i,w)+1 & \text{if } u \in V(H_i) \text{ and } w \in V(G)
\\ d_G(u,w) & \text{if } u,w \in V(G)

\end{cases} 
\end{equation} 

The remaining definitions and notation used throughout the paper will be introduced in the context in which they are applied.
For undefined terms we refer the reader to \cite{West2001}.

The relationship between independent sets and vertex covers in a graph $G$ will play a central
role in establishing tight bounds for the \emph{equidistant dimension} of corona product graphs. For this reason, we recall the following classical result.

\begin{theorem}[Gallai's Theorem \cite{Gallai1959}]\label{Gallai} For any graph $G$, $$\alpha(G)+\beta(G) = \n(G).$$
\end{theorem}

\section{Main Results} \label{Sec:MainResults}

Note that for every graph $H$ we have $\xi(K_1 \odot H) = 1$, as the only vertex of $K_1$ is universal in $K_1 \odot H$. Therefore, from now on, we shall consider simple and connected graphs $G$ of order at least 2.

\begin{lemma}\label{lemmaLowerBound}

Let $G$ and $H$ be two graphs, with $G$ connected. For every distance-equalizer set $S$ of $G\odot H$ we have 
\[ S\cap(\{v_i\}\cup V(H_i))\neq \varnothing  , \; \text{for all } i\in[\n(G)],\] 
and as a consequence, $\xi(G \odot H) \geq \n(G) $
\end{lemma}

\begin{proof}

If $\n(G)=1$, the result is immediate, as $V(G \odot H)=\{v_1\}\cup V(H_1)$. 
Otherwise, let $S \subseteq V(G\odot H)$ be a distance-equalizer set of $G\odot H$. Since for any $i \in [\n(G)], h \in V(H_i)$ and $u \in V(G \odot H)\setminus(\{v_i\}\cup V(H_i))$ we have $d(u,h) = d(u,v_i)+1$, it follows that $S\cap(\{v_i\}\cup V(H_i))\neq \varnothing $.

\end{proof}

\begin{lemma}\label{includedHi}

Let $G$ and $H$ be two graphs, with $G$ connected. If $S$ is a distance-equalizer set of $G\odot H$, then for any pair of distinct vertices $v_i, v_j \in V(G)$ such that $B_G(v_i \mid v_j) = \varnothing$, the following conditions hold:
\begin{enumerate}[\rm (i)]
\item $v_i \in S$ or $v_j\in S$
\item $V(H_{i})\subseteq S$ or $V(H_{j})\subseteq S$
\end{enumerate}

\end{lemma}

\begin{proof}
Let $i,j \in [\n(G)]$ be such that $i\neq j$ and $B_G(v_i\mid v_j) = \varnothing$. If $\{v_i, v_j\}\cap S = \varnothing $, then $S$ cannot be a distance-equalizer set of $G \odot H$, which implies that $v_i \in S$ or $v_j \in S$. Additionally, $B_{G\odot H}(u \mid w) = \varnothing $ for any $u \in V(H_i)$ and any $w \in V(H_j)$; hence, $V(H_i)\subseteq S$ or $V(H_j)\subseteq S$.
\end{proof}

\begin{definition} \label{def_LU}
Let $G$ and $H$ be two graphs, with $G$ connected. 
For any set $S \subseteq V(G\odot H)$, we define the following two operators:

\begin{enumerate}[\rm (i)]
\item $L(S) = S \cap V(G),$
\item $U(S) = \{ v_i \in V(G) : V(H_i) \cap S \ne \varnothing \}.$
\end{enumerate}

\end{definition}

In what follows, for any graph $G$, we denote by $\mathcal{B}(G)$ the set of all vertex covers of $G$.

According to the definition of $\widehat{G}$, for any pair of vertices $v_i, v_j \in V(G)$ satisfying $ \{v_i, v_j\} \in E(\widehat{G})$, we have $B(v_i \mid v_j) = \varnothing$. Therefore, considering Definition \ref{def_LU}, we can deduce the following result from Lemmas \ref{lemmaLowerBound} and \ref{includedHi}.

\begin{lemma} \label{vertxcovsets}
Let $G$ and $H$ be two graphs, with $G$ connected. 
If $S$ is a distance-equalizer set of $G\odot H$, then 
\begin{enumerate}[{\rm (i)}]
\item $L(S) \in \mathcal{B}(\widehat{G})$,
\item $U(S) \in \mathcal{B}(\widehat{G})$,
\item $L(S) \cup U(S) = V(G).$
\end{enumerate}
\end{lemma}
Observe that, 
$|L(S)| \ge \beta(\widehat{G})$ and $|U(S)| \ge \beta(\widehat{G})$, since $L(S)$ and $U(S)$ are vertex covers of $\widehat{G}$.

The next definition plays a key role in the characterization of distance-equalizer sets satisfying certain convenient conditions in corona product graphs.

\begin{definition} \label{Def:ForwardResvPair}
Let $ G $ be a graph and let $ X, Y \subseteq V(G) $ such that $X\cup Y = V(G)$.
We say that $ (X, Y) $ is a \emph{forward-equalized pair of $G$}, if for every pair of vertices
\[
(u, v) \in (X \setminus Y) \times (Y \setminus X),
\] 
there exists a vertex $ w \in V(G) $ such that
\[
d_G(w, u) = d_G(w, v) + 1.
\]
\end{definition}

Note that, for every graph $G$ we can always construct a forward-equalized pair $(X,Y)$ by taking any $X \subseteq V(G)$ and $Y = V(G)$. Hence, a forward-equalized pair always exists for every graph $G$. 

To further illustrate Definition \ref{Def:ForwardResvPair}, we provide an example. In Figure~\ref{Fig:ForwardResvPair}, the sets of vertices $X_1 =\{1,2,3,5,6\}$ and $Y_1 = \{1,2,3,4\}$ form a forward-equalized pair $(X_1,Y_1)$; however, $(Y_1,X_1)$ is not a forward-equalized pair of $G$.

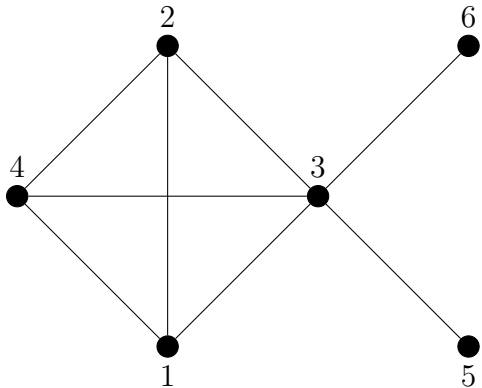
\begin{figure}[h]
  \centering
\begin{tikzpicture}[transform shape, inner sep = 1mm]
	\node[ draw, circle, fill=black, label=above:3] (v3) at (0,0) {};
	\node[ draw, circle, fill=black, label=above:6] (v6) at (2,2) {};
	\node[draw, circle, fill=black, label=below:1] (v1) at (-2,-2) {};
	\node[ draw, circle, fill=black, label=above:4] (v4) at (-4,0) {};
	\node[ draw, circle, fill=black, label=below:5] (v5) at (2,-2) {};
	\node[ draw, circle, fill=black, label=above:2] (v2) at (-2,2) {};
	\draw (v3) -- (v6);
	\draw (v3) -- (v1);
	\draw (v3) -- (v5);
	\draw (v3) -- (v2);
	\draw (v4) -- (v2);
	\draw (v4) -- (v1);
	\draw (v2) -- (v1);
	\draw (v4) -- (v3);
\end{tikzpicture}
\caption{Graph $G$ used to illustrate forward-equalized pairs.}
\label{Fig:ForwardResvPair}
\end{figure}

\begin{theorem} \label{disteqGenralcond}
Let $G$, $H$ be two graphs, with $G$ connected, and let $X, Y$ be two vertex covers of $\widehat{G}$.
The set $$ S = \left( \bigcup \limits_{v_i \in X} V(H_i) \right) \cup Y $$
is a distance-equalizer set of $G \odot H$, if and only if, $(X,Y)$ is a forward-equalized pair of $G$.
\end{theorem}

\begin{proof} 

(Necessity) Note that $U(S)=X$ and $L(S)=Y$. Since $S$ is a distance equalizer set of $G \odot H$, by Lemma \ref{vertxcovsets} it follows that $X \cup
Y=V(G)$. Additionally, if $ u \in X \setminus Y $ and $ v = v_i \in Y \setminus X $, then $ u, h_i \notin S $ for any $ h_i \in V(H_i) $. Moreover, there exists some $ j \in [\n(G)] $ and a vertex $ w \in \big( \{v_j\} \cup V(H_j) \big) \cap S $ such that
$$
d_{G \odot H}(u, w) = d_{G \odot H}(w, h_i).
$$
Consequently, $ v_j \in B_{G \odot H}(u \mid h_i) $, and we obtain
$$
d_G(u, v_j) = d_{G \odot H}(u, v_j) = d_{G \odot H}(v_j, h_i)
= d_G(v_j, v_i) + d_{G \odot H}(v_i, h_i) = d_G(v_j, v_i) + 1, 
$$ as required.

Therefore, $(X,Y)$ is a forward-equalized pair of $G$.

\medskip
(Sufficiency) We differentiate all possible cases for any two vertices $u,v \in V(G\odot H)\setminus S$:
\begin{enumerate}[{Case }1:]
\item $u,v \in V(H_i)$. Note that this case is only possible if $\n(H)\geq 2$ and $v_i \in Y \subseteq S $. Thus $S\cap B_{G\odot H}(u\mid v)\neq\varnothing $.

\item $u \in V(H_i)$ and $v\in V(H_j)$, with $i\ne j$. It should be noted that this case can occur only if $v_i, v_j\in Y \setminus X$. 
Since $Y \setminus X$ is an independent set in $\widehat{G}$, we have $B_G(v_i\mid v_j)\ne\varnothing $, and as a consequence, there exists a vertex $v_k\in B_{G\odot H}(v_i\mid v_j)$. Since $\{v_k\}\cup V(H_k) \subseteq B_{G\odot H}(v_i\mid v_j) = B_{G\odot H}(u\mid v)$ and $S \cap (\{v_k\}\cup V(H_k))\neq \varnothing $, then we get $S\cap B_{G\odot H}(u\mid v)\neq\varnothing $.

\item $u,v\in V(G)$. This case is possible only when $u,v\in X \setminus Y$. Since $X \setminus Y$ is an independent set in $\widehat{G}$, there exists a vertex $v_k\in B_G(u\mid v)$. Analogously to the previous case, we have $S\cap B_{G\odot H}(u\mid v)\neq\varnothing  $.

\item $u \in V(G)$ and $v \in V(H_j)$. Observe that this scenario is only feasible if $u \in X \setminus Y$ and $v_j \in Y \setminus X$. Since $(X,Y)$ is a forward-equalized pair of $G$,  there exists a vertex $v_k \in V(G)$ such that $d_{G\odot H}(u, v_k)=d_G(u, v_k) = d_G(v_k,v_j)+1= d_{G\odot H}(v_k,v_j)+d_{G\odot H}(v_j,v)=d_{G\odot H}(v_k,v)$. Hence, $v_k \in B_{G\odot H}(u \mid v)$. Analogously to Case 2, we can conclude that $S\cap B_{G\odot H}(u\mid v)\neq\varnothing $.

\end{enumerate}
 
Based on the preceding cases, we conclude that $S$ is a distance-equalizer set of $ G \odot H $.

\end{proof}

The following result aims to provide a well-structured $\xi(G \odot H)$-set, which simplifies the computation of $\xi(G \odot H)$. 

\begin{lemma} \label{structHindep}
For any two graphs $G$ and $H$, with $G$ connected, there exists a $\xi(G \odot H)$-set $S$, such that $V(H_i) \cap S = \varnothing $ or $ V(H_i) \subseteq S $, for every $i \in [\n(G)]$.
\end{lemma}

\begin{proof}

Let $S \subseteq V(G \odot H)$ be such that $|S \setminus V(G)|$ is minimum among all $\xi(G \odot H)$-sets $S$. 

For $\n(H) = 1$,  $V(H_i) \cap S \neq \varnothing $ implies $ V(H_i) \subseteq S$, and the statement holds.

Otherwise, if $\n(H) \geq 2$, suppose there exist $u \in V(H_i) \cap S$ and $w \in V(H_i) \setminus S$, for some $i \in [\n(G)]$. 
Let us define the set $S' \subseteq V(G \odot H)$ in the following way:
\[ S' = (S\cup \{v_i\})\setminus\{u\}. \]

We claim that $S'$ is a distance-equalizer set of $G \odot H$. To prove this, let us consider the possible cases for any pair of different vertices $x, y \in V(G \odot H) \setminus S'$.
\begin{enumerate}[{Case }1:]

\item $x,y \in V(H_i)$. As $v_i \in S'$, we have $B_{G\odot H}(x \mid y)\cap S' \neq \varnothing $.

\item $x, y \in V(G \odot H) \setminus V(H_i) $. Note that, if $u \in B_{G\odot H}(x \mid y)$, then  $v_i \in B_{G\odot H}(x \mid y)\cap S'$. Otherwise, if $u \notin B_{G\odot H}(x \mid y)$, then $v_i \notin B_{G\odot H}(x \mid y)$. Consequently, $B_{G\odot H}(x \mid y) \cap S' = B_{G\odot H}(x \mid y) \cap S$. Since $S$ is a distance-equalizer set of $G\odot H$, we have  $B_{G\odot H}(x \mid y) \cap S' \neq \varnothing $.

\item $x \in V(H_i)$ and $y \in V(G \odot H)\setminus V(H_i) $. If $ d_{G\odot H}(x,y)=2$, then $v_i \in B_{G\odot H}(x \mid y) \cap S'$. Otherwise, if $ d_{G\odot H}(x,y) > 2$, then $B_{G\odot H}(x \mid y) \cap V(H_i) = \varnothing$ and $B_{G\odot H}(x \mid y) = B_{G\odot H}(w \mid y)$. Consequently, $B_{G\odot H}(x \mid y) \cap S' \neq \varnothing $, as $B_{G\odot H}(w \mid y) \cap (S \setminus \{u\}) \neq \varnothing $. 

\end{enumerate}

As a result, for every $x, y \in V(G \odot H) \setminus S'$, it holds that $B_{G\odot H}(x \mid y) \cap S' \neq \varnothing $. This confirms that $S'$ is a distance-equalizer set of $G \odot H$.

Next, consider two cases. If $v_i \in S$, we have $|S'| = |S| - 1 < \xi(G \odot H)$, which yields a contradiction. Conversely, if $v_i \notin S$, then $|S' \setminus V(G)| < |S \setminus V(G)|$, which contradicts the minimality assumption on $S$.
Therefore, the result follows. 
\end{proof}

It should be noted that for any distance-equalizer set $S$ of $G \odot H$, satisfying conditions of Lemma~\ref{structHindep}, we have
$$S =\left( \bigcup\limits_{v_i \in U(S)} V(H_i) \right) \cup L(S) \qquad \text{and} \qquad |S| = |U(S)|\n(H) + |L(S)|. $$ 

Nevertheless, not every distance-equalizer set satisfies Lemma~\ref{structHindep}. As it is shown in Figure \ref{fig:C3xP2}, for the case of the corona product graph $C_3 \odot P_2$, we can construct a $\xi(C_3 \odot P_2)$-set, by  taking only one vertex from one of the copies of $P_2$ and two other vertices from de base graph $C_3$. 

\begin{figure}[h!]
  \centering
  \begin{tikzpicture}[transform shape, inner sep=1mm]
    
    \tikzset{
      nodesBlk/.style={draw=black, circle, fill=black, minimum size=8pt, inner sep=0pt, label distance=2.5mm},
      nodesWht/.style={draw=black, circle, fill=white, minimum size=8pt, inner sep=0pt, label distance=2.5mm}
    }

    \def\radiusInner{1.4}   
    \def\radiusOuter{2.4}   
    \def\attachOffset{18}   
    \def\labDistInner{0.8 mm} 

    \node[nodesWht, label={[label distance=\labDistInner]below:1}] (c1) at (90:\radiusInner cm) {};
    \node[nodesBlk, label={[label distance=\labDistInner]above:2}] (c2) at (210:\radiusInner cm) {};
    \node[nodesBlk, label={[label distance=\labDistInner]above:3}]  (c3) at (330:\radiusInner cm) {};

    \draw[black] (c1) -- (c2);
    \draw[black] (c2) -- (c3);
    \draw[black] (c3) -- (c1);

    \node[nodesBlk,  label=above:4] (p11) at ({90-\attachOffset}:\radiusOuter cm) {};
    \node[nodesWht, label=above:5] (p12) at ({90+\attachOffset}:\radiusOuter cm) {};
    \draw[black] (c1) -- (p11);
    \draw[black] (c1) -- (p12);
    \draw[black] (p11) -- (p12);

    \node[nodesWht,  label=below left:6] (p21) at ({210-\attachOffset}:\radiusOuter cm) {};
    \node[nodesWht, label=left:7]      (p22) at ({210+\attachOffset}:\radiusOuter cm) {};
    \draw[black] (c2) -- (p21);
    \draw[black] (c2) -- (p22);
    \draw[black] (p21) -- (p22);

    \node[nodesWht,  label=below right:8] (p31) at ({330-\attachOffset}:\radiusOuter cm) {};
    \node[nodesWht, label=right:9]       (p32) at ({330+\attachOffset}:\radiusOuter cm) {};
    \draw[black] (c3) -- (p31);
    \draw[black] (c3) -- (p32);
    \draw[black] (p31) -- (p32);

  \end{tikzpicture}
\caption{Corona product $C_3\odot P_2$, for which there exists a $\xi(C_3\odot P_2)$-set $S=\{2,3,4\}$ that does not satisfy Lemma~\ref{structHindep}.}
  \label{fig:C3xP2}
\end{figure}
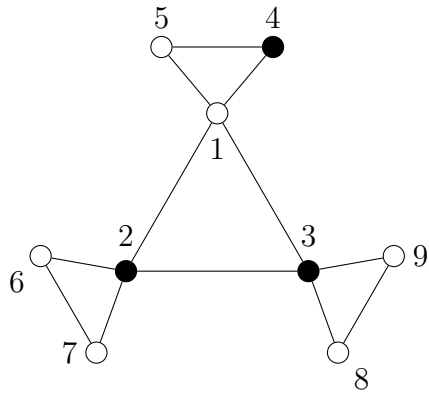

Lemma~\ref{structHindep} implies there always exists a $\xi(G \odot H)$-set which is independent of the structure of the graph $H$. Hence, the only property of $H$ that is relevant for determining $\xi(G \odot H)$ is the order.

\begin{proposition}
Let $G$ be a connected graph. If $H$ and $H'$ are two graphs such that $\n(H)=\n(H')$, then
$$ \xi(G \odot H) = \xi(G \odot H').$$
\end{proposition}

\begin{proof}
Let $S$ be a $\xi(G \odot H)$-set that satisfies the conditions of Lemma~\ref{structHindep}. Applying Lemma \ref{lemmaLowerBound} and Lemma \ref{vertxcovsets}, we have $L(S) \cup U(S) = V(G)$, and $L(S), U(S)$ are vertex covers of $\widehat{G}$.
Moreover, by Theorem \ref{disteqGenralcond},
$$S' =\left( \bigcup\limits_{v_i \in U(S)} V(H^{'}_i) \right) \cup L(S) $$ 
is a distance-equalizer set of $G \odot H'$.
Thus, $\xi(G \odot H) = |S| = |U(S)|\n(H)+|L(S)| = |S'| \ge \xi(G \odot H')$.

Analogously, if we take a $\xi(G \odot H')$-set  satisfying the conditions of Lemma~\ref{structHindep} we can deduce that $\xi(G \odot H') \ge \xi(G \odot H)$, and consequently the equality holds.

\end{proof}

We introduce the following parameter, which will be used to establish a sharp lower bound for the equidistant dimension of corona product graphs.

\begin{definition} \label{intersectminEq}
For any connected graph $G$, we define $$ \beta^*(G) = \min \limits_{X,Y \in \mathcal{B}(\widehat{G})} \{ |X \cap Y|:  (X,Y) \text{ is a forward-equalized pair of } G \}$$
\end{definition}

\begin{theorem} \label{improvedlbound}
For any two graphs $G$ and $H$, with $G$ connected, $$\xi(G \odot H) \ge \beta(\widehat{G})\n(H) + \alpha(\widehat{G}) + \beta^*(G).$$
\end{theorem} 

\begin{proof}

Let $S$ be a $\xi(G \odot H)$-set, which satisfies the conditions of Lemma~\ref{structHindep}. We recall that 
$$ S = \left( \bigcup_{v_i \in U(S)} V(H_i) \right) \cup L(S).$$
Hence, by Theorem \ref{disteqGenralcond}, $(U(S), L(S))$ is a forward-equalized pair of $G$. Since $|U(S) \cap L(S)| \geq \beta^*(G)$, we obtain:
\begin{align*}
|S| & = |U(S)|\n(H)+|L(S)| \\ 
& = \beta(\widehat{G}) \n(H) + \left(|U(S)| -\beta(\widehat{G}) \right)\n(H)+|L(S)| \\ 
& \ge \beta(\widehat{G})\n(H) + |U(S)| - \beta(\widehat{G})  + |L(S)| \\
& = \beta(\widehat{G})\n(H) + \n(\widehat{G}) - \beta(\widehat{G}) + |U(S) \cap L(S)|\\
& = \beta(\widehat{G})\n(H) + \alpha(\widehat{G}) + |U(S) \cap L(S)| \\
& \geq \beta(\widehat{G})\n(H) + \alpha(\widehat{G}) + \beta^*(G)
\end{align*}
Therefore, the result follows.

\end{proof}

Although straightforward, the corollary of Theorem~\ref{improvedlbound} stated below is particularly convenient for the subsequent analysis, as it simplifies most arguments by allowing us to avoid any reference to $\beta^*(G)$.

\begin{corollary} \label{lowerbipartite}
For any two graphs $G$ and $H$, with $G$ connected, $$\xi(G \odot H) \geq \beta(\widehat{G})\n(H) + \alpha(\widehat{G}).$$
\end{corollary}

The lower bound proposed in Theorem~\ref{improvedlbound} is tight. In fact,
Theorem~\ref{eqdimbipartG}, stated below, shows that when the base graph $G$ is bipartite, the equidistant dimension coincides with this bound. Moreover, this lower bound is also achieved for some other cases where $G$ is not a bipartite graph, as will be illustrated in the example of Figure \ref{fig:Fig2}. In fact, the following result provides a sufficient condition on the graph $G$ under which the lower bound given in Theorem~\ref{improvedlbound} is attained.

\begin{proposition} \label{lowerbound_eq_condition}
Let $G$ and $H$ two graphs, with $G$ connected. If there exist $ U, L \in \mathcal{B}(\widehat{G})$ such that $(U,L)$ is a forward-equalized pair of $G$, with $|U \cap L| = \beta^*(G)$ and $|U|= \beta(\widehat{G})$, 
then $$\xi(G \odot H)= \beta(\widehat{G})\n(H) + \alpha(\widehat{G}) + \beta^*(G).$$ 
\end{proposition}

\begin{proof}
Suppose there exist $ U, L \in \mathcal{B}(\widehat{G})$ such that $(U,L)$ is a forward-equalized pair of $G$, with $|U \cap L| = \beta^*(G)$ and $|U|= \beta(\widehat{G})$. By Theorem~\ref{disteqGenralcond}, $$ S = \left( \bigcup_{v_i \in U} V(H_i) \right) \cup L$$ is a distance-equalizer set of $G \odot H$. 
Thus,
\begin{align*}
|S| & = |U|\n(H)+|L| \\  
& = \beta(\widehat{G})\n(H) + |L \setminus U|  + |L| \\
& = \beta(\widehat{G})\n(H) + \n(\widehat{G}) - \beta(\widehat{G}) + |U \cap L|\\
& = \beta(\widehat{G})\n(H) + \alpha(\widehat{G}) + |U \cap L| \\
& = \beta(\widehat{G})\n(H) + \alpha(\widehat{G}) + \beta^*(G)
\end{align*} 

Now, applying the lower bound stated in Theorem~\ref{improvedlbound} we have $$\beta(\widehat{G})\n(H) + \alpha(\widehat{G}) + \beta^*(G) \le \xi(G \odot H) \le |S|.$$ Implying that $$\xi(G \odot H)= \beta(\widehat{G})\n(H) + \alpha(\widehat{G}) + \beta^*(G).$$
\end{proof}

In Figure \ref{fig:Fig2}, we present an example of a non-bipartite connected graph $G$, for which the conditions of Proposition \ref{lowerbound_eq_condition} are satisfied. 
Thus, the lower bound of Theorem~\ref{improvedlbound} is achieved.
In this case the sets $L = \{1, 3, 4, 6\}$ and $U = \{2, 4, 5, 7\}$ are $\beta(\widehat{G})$-sets, and $(U,L)$ is a forward-equalized pair of $G$, with $|L \cap U|$ minimum. Hence, $\beta(\widehat{G})=|U|=4$, $\beta^*(G) =|L \cap U|= 1$ and $\alpha(\widehat{G})=3$. Therefore, applying Proposition~\ref{lowerbound_eq_condition}, we get $\xi(G \odot H)= \beta(\widehat{G})\n(H) + \alpha(\widehat{G}) + \beta^*(G) = 4\cdot \n(H)  + 4$, for any graph $H$. 

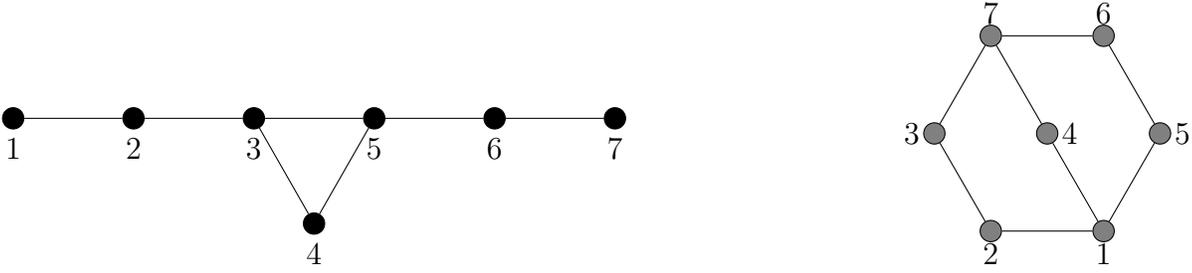
\begin{figure}[h!]
  \centering
  \begin{subfigure}{0.45\textwidth}
    \centering
	\begin{tikzpicture}[transform shape, inner sep = 1mm]
\foreach \x in {1,2,3}
{
  \pgfmathtruncatemacro{\iind}{\x};
  \pgfmathsetmacro{\xc}{(\x-1)*1.6}; 
  \node [draw=black, shape=circle, fill=black] (v\iind) at (\xc,0) {};
  \node at ([yshift=-.4 cm]v\iind) {$\iind$};
}
\foreach \x in {5,6,7}
{
  \pgfmathtruncatemacro{\iind}{\x};
  \pgfmathsetmacro{\xc}{(\x-2)*1.6}; 
  \node [draw=black, shape=circle, fill=black] (v\iind) at (\xc,0) {};
  \node at ([yshift=-.4 cm]v\iind) {$\iind$};
}
\pgfmathsetmacro{\xc}{2.5*1.6};
\pgfmathsetmacro{\yc}{-1.4};
\node [draw=black, shape=circle, fill=black] (v4) at (\xc, \yc) {};
\node at ([yshift=-.4 cm]v4) {$4$};
\foreach \ind in {1,...,6}
{
\pgfmathparse{int(\ind+1)}
\draw[black] (v\ind) -- (v\pgfmathresult);
}
\draw[black] (v3) -- (v5);
\end{tikzpicture}
\end{subfigure}
  \hfill
  \begin{subfigure}{0.45\textwidth}
    \centering
\begin{tikzpicture}[transform shape, inner sep = 1mm]
\def\radius{1.5} 
\foreach \ind in {1,...,6}
{
\pgfmathparse{360/6*\ind};
\node [draw=black, shape=circle, fill=gray] (v\ind) at (\pgfmathresult:\radius cm) {};
}

\node [draw=black, shape=circle, fill=gray] (v7) at (0,0 cm) {};
\node at ([xshift=.3 cm]v7) {4};
\node  at ([yshift=.3 cm]v1) {6};
\node  at ([yshift=.3 cm]v2) {7};
\node  at ([xshift=-.3 cm]v3) {3};
\node  at ([yshift=-.3 cm]v4) {2};
\node  at ([yshift=-.3 cm]v5) {1};
\node  at ([xshift=.3 cm]v6) {5};

\foreach \ind in {1,...,5}
\pgfmathparse{int(\ind+1)}
\draw[black] (v\ind) -- (v\pgfmathresult);

\draw[black] (v1) -- (v6);
\draw[black] (v2) -- (v7);
\draw[black] (v7) -- (v5);

\end{tikzpicture}

\end{subfigure}
  \caption{A non-bipartite graph $G$ (to the left) with $\beta^*(G) > 0$, and its corresponding empty bisector graph $\widehat{G}$ (to the right). In this example the lower bound stated in Theorem~\ref{improvedlbound} is attained.}
   \label{fig:Fig2}
\end{figure}

Since for any $\beta(\widehat{G})$-set $U$, the combination $\big(U,V(G)\big)$ constitutes a forward- equalized pair of $G$, we can establish the following relationship between  $\beta^*(G)$ and $\beta(\widehat{G})$.

\begin{remark} \label{b_aster_bound}
For any connected graph $G$, 
$$ \beta^*(G) \le \beta(\widehat{G}).$$
\end{remark}

\begin{theorem} \label{UpperBound1}
If $G$ and $H$ are two graphs such that $G$ is connected, then  $$\xi(G \odot H) \leq \beta(\widehat{G})\n(H)+ \n(G).$$
\end{theorem}

\begin{proof}
Let $U$ be a $\beta(\widehat{G})$- set, and let $L = V(G)$. Observe that $U$ and $L$ are vertex covers of $\widehat{G}$ and $(U,L)$ is a forward-equalized pair of $G$. Consequently, by Theorem \ref{disteqGenralcond} the set
$$S = \left( \bigcup \limits_{v_i \in U} V(H_i) \right) \cup L$$
is a distance-equalizer set of $G \odot H$. 

Now we have, $$\xi(G \odot H) \leq \ |S| =  \beta(\widehat{G})\n(H)+ \n(G),$$
which proves the claim.

\end{proof}

The next result gives a characterization under which the equidistant dimension of the corona product graph $G \odot H$ takes its minimum possible value, $\n(G)$. In fact, while this result ensures that the upper bound in Theorem~\ref{UpperBound1} is tight when $\beta(\widehat{G}) = 0$, Figure~\ref{fig:FigFish} illustrates that this bound can also be attained when $\beta(\widehat{G}) > 0$.

\begin{theorem}\label{totalEq}
If $G$ is a connected graph and $H$ any graph, then $\xi(G \odot H) = \n(G)$ if and only if either $\beta(\widehat{G}) = 0$, or $(\n(H)=1$ and $\beta^*(G)=0)$.
\end{theorem}

\begin{proof}
(Necessity) Assume that $\xi(G \odot H)=\n(G)$. By Theorem~\ref{improvedlbound} and Gallai's Theorem~\ref{Gallai}, we obtain
\begin{align*}
\beta(\widehat{G})\,\n(H) + \alpha(\widehat{G}) + \beta^*(G) &\le \n(G) \\
&\le \beta(\widehat{G}) + \alpha(\widehat{G}).
\end{align*} 
Consequently,
\[
\beta(\widehat{G})\bigl(\n(H)-1\bigr) + \beta^*(G) \le 0.
\]

Since all terms involved are nonnegative, the above inequality implies that
\[
(\beta(\widehat{G})=0 \text{ and } \beta^*(G)=0) \quad \text{or} \quad (\n(H)=1  \text{ and } \beta^*(G)=0).
\]

By Remark \ref{b_aster_bound}, if $\beta(\widehat{G}) = 0$, then $\beta^*(G)=0$. Therefore, we conclude that either $\beta(\widehat{G})=0$, or $\n(H)=1$ and $\beta^*(G)=0$.

(Sufficiency) If $\beta(\widehat{G})=0$, then applying the lower bound stated in Lemma \ref{lemmaLowerBound}, together with the upper bound stated in Theorem~\ref{UpperBound1}, the result follows. 

Now, if $\beta^*(G)=0$, then there exists a forward-equalized pair $(U,L)$ such that $U,L\in\mathcal{B}(\widehat{G})$, $U\cup L=V(G)$, and $|U\cap L|=\beta^*(G)=0$. By Theorem~\ref{disteqGenralcond}, the set
$$S = \left( \bigcup \limits_{v_i \in U} V(H_i) \right) \cup L$$
is a distance-equalizer set of $G \odot H$. In particular, if $\n(H)=1$, then $|S|=\n(G)\ge \xi(G \odot H)$. Therefore, applying the lower bound from Lemma~\ref{lemmaLowerBound}, we conclude that $\xi(G \odot H) = \n(G)$.
\end{proof}

We point out that, $\beta^*(G)=0$ does not implies $\beta(\widehat{G})=0$.  
Consider the particular case in which the base graph $G$ is isomorphic to $K_5 \odot N_1$. As illustrated in Figure~\ref{fig:Fig3}, in this case, the empty bisector graph $\widehat{G}$ is isomorphic to $N_5 \odot N_1$. 
Moreover, since $\left(V(K_5 \odot N_1) \setminus V(K_5), V(K_5) \right)$ form a disjoint forward-equalized pair of $G$, it follows that $\beta^*(G)=0$, despite the fact that $\beta(\widehat{G})=5$.

The \emph{total equidistant dimension} of a graph, denoted by $\xi_t(G)$, was recently introduced in \cite{GispertFernandez2024}. This parameter provides a tool to characterize graphs for which every pair of vertices has a non empty bisector. In particular, $\xi_{t}(G)<+\infty$ is satisfied if and only if $\widehat{G}$ is an empty graph, which is equivalent to the condition $\beta(\widehat{G})=0$.
In Section~\ref{Sec:Computation}, we compute $\xi(G \odot H)$ for several families of base graphs $G$ satisfying this property.

We recall that, given two simple graphs $G_1$ and $G_2$  with disjoint vertex sets. 
The \emph{sum} (or \emph{join}) of $G_1$ and $G_2$, denoted by $G_1+G_2$, is the graph 
with vertex set 
\[
V(G_1+G_2)=V(G_1)\cup V(G_2)
\]
and edge set
\[
E(G_1+G_2)=E(G_1)\cup E(G_2)\cup \{\, \{u,v\} : u\in V(G_1),\ v\in V(G_2) \,\}.
\]
That is, $G_1+G_2$ is obtained from the disjoint union of $G_1$ and $G_2$ by adding all possible edges between every vertex of $G_1$ and every vertex of $G_2$.

The following result is a direct consequence of Theorem~\ref{totalEq} and highlights its applicability.

\begin{proposition} \label{eqdim_sumgraphs}
If $G_1$ and $G_2$ are two graphs such that at least one of them has no isolated vertices, then for any graph $H$
$$ \xi \big((G_1 + G_2) \odot H \big) = \n(G_1) + \n(G_2).$$
\end{proposition}

\begin{proof}
We aim to prove that $\beta\big(\widehat{G_1 + G_2} \big) = 0$.  
Without loss of generality let us assume that $G_1$ has no isolated vertex. For any pair of vertices $u,v \in V(G_1 + G_2) $ we differentiate all possible cases:

\begin{enumerate}[{Case }1:]
\item If $u \in V(G_1)$ and $v \in V(G_2)$, then, since $G_1$ has no isolated vertices, there exists a vertex $w \in N_{G_1}(u) $. Hence, we have $d_{G_1 + G_2}(w,u)=1=d_{G_1 + G_2}(w,v)$.
\item If $u \in V(G_1)$ and $v \in V(G_1)$, then for any vertex $w \in V(G_2)$, it holds $d_{G_1 + G_2}(w,u)=1=d_{G_1 + G_2}(w,v)$.
\item If $u \in V(G_2)$ and $v \in V(G_2)$, then for any vertex $w \in V(G_1)$, it holds $d_{G_1 + G_2}(w,u)=1=d_{G_1 + G_2}(w,v)$.
\end{enumerate}
In any case $B_{G_1 + G_2}(u \mid v) \ne \varnothing$. Therefore the empty bisector graph $\widehat{G_1 + G_2}$ has no edges, which is equivalent to $\beta\big(\widehat{G_1 + G_2} \big) = 0$. Now applying  Theorem~\ref{totalEq}, we get
$$ \xi \big((G_1 + G_2) \odot H \big) = \n(G_1 + G_2) = \n(G_1) + \n(G_2).$$
\end{proof}

The following three results lead to the computation of the equidistant dimension in the case where $G$ is a bipartite graph.

\begin{proposition} \label{form_coroncg}
Let $G$ and $H$ be two graphs such that $G$ is connected. If the empty bisector graph $\widehat{G}$ is a complete bipartite graph, then
$$\xi(G\odot H) = \beta(\widehat{G})\n(H)+ \alpha(\widehat{G})$$
\end{proposition}

\begin{proof}

Let $L$ and $U$ be the partite sets of $\widehat{G}$, with $|L| \ge |U|$.
Since $\widehat{G}$ is a complete bipartite graph, we have that $L$ is an $\alpha(\widehat{G})$-set and $U$ is a $\beta(\widehat{G})$-set. Additionally, $L$ and $U$ are vertex covers of $\widehat{G}$.

Consider the set
\[
S = \left( \bigcup_{v_i \in U} V(H_i) \right) \cup L.
\]

For any $u \in U$ and $v \in L$, we have $u \in N_{\widehat{G}}(v)$. Hence, $B_G(v \mid v) = \varnothing$, and consequently, $d_G(u,v)$ is odd. Thus, there exists a vertex $w$ on a shortest path between $u$ and $v$ such that $d_G(u,w) = d_G(w,v) + 1$. Therefore, $(U, L)$ is a forward-equalized set of $G$.

Applying Theorem~\ref{disteqGenralcond}, it follows that $S$ is a distance-equalizer set of $G \odot H$. Thus,
\[
\xi(G \odot H) \le |S| = \beta(\widehat{G})\n(H) + \alpha(\widehat{G}).
\]

Furthermore, by Corollary \ref{lowerbipartite}, we have the corresponding lower bound
\[
\xi(G \odot H) \ge \beta(\widehat{G})\n(H) + \alpha(\widehat{G}).
\]

As a consequence, we conclude that
\[
\xi(G \odot H) = \beta(\widehat{G})\n(H) + \alpha(\widehat{G}).
\]
\end{proof}

\begin{lemma} \label{complete_bipart}
If $G$ is a bipartite connected graph, then the corresponding empty bisector graph $\widehat{G}$ is  complete bipartite.
\end{lemma}

\begin{proof}
Let $A$ and $B$ be the partite sets of $G$. We will distinguish two cases for any two different vertices $u$, $v \in V(G)$:

\begin{enumerate}[{Case }1:] \item $u \in A$ and $v\in B$. Suppose that $B_G(u \mid v) \neq \varnothing $, then there exists a path of even length between $u$ and $v$, in contradiction with $G$ being a bipartite graph. Hence, $B_G(u \mid v) = \varnothing $ and $\{u,v\}\in E(\widehat{G})$.

\item $u,v \in A$ or $u,v \in B$. The distance between two vertices in the same partite set is even, consequently there exists a vertex $w \in V(G)$ such that $w\in B_G(u \mid v)$ which leads to $\{u,v\}\notin E(\widehat{G})$.
\end{enumerate}

Therefore, $A$ and $B$ are partite sets of $\widehat{G}$ and every pair of vertices in different partite sets are adjacent in $\widehat{G}$.
\end{proof}

Considering that every connected bipartite graph has unique partite sets, together with Proposition~\ref{form_coroncg} and Lemma~\ref{complete_bipart}, we arrive at the following conclusion.

\begin{theorem} \label{eqdimbipartG}
If $G$ is a connected bipartite graph, with partite sets $A,B$ such that $|A| \geq |B|$, then for any graph $H$, $$\xi(G\odot H) = |B|\n(H)+ |A|.$$
\end{theorem}

\subsection{Computation of $\xi(G \odot H)$ for some families of graphs $G$} \label{Sec:Computation}

The preceding results allow us to determine the  equidistant dimension of $G \odot H$ for several important families of connected graphs $G$ and for any graph $H$. We begin by recalling some necessary notation.

We denote by $K_{n_1,\ldots,n_p}$ the \emph{complete $p$-partite graphs}, with vertex set $A_1 \cup \cdots \cup A_p$ such that $A_1, \ldots, A_p$, are pairwise disjoint, verifying $|A_i| = n_i \geq 1$. Finally the \emph{bistar graphs} are denoted by $K_2(r,s)$ and
the \emph{hypercube} of dimension $n$ is denoted by $Q_n$.

\begin{proposition} \label{familieseqidim}

Let $n, r, s, p, n_1, \ldots, n_p$ be positive integers such that $n \geq 2, s \geq r, p \geq 3$ and $n_p \geq \cdots \geq n_1 \geq 1$. For every graph $H$ the following statements hold:

\begin{enumerate}[{\rm (i)}]
\item $\xi(K_n \odot H) = \left\lbrace
\begin{array}{cl}
  n, & \text{if } n \neq 2
\\ \n(H) + 1, & \text{if } n = 2
\end{array} \right. $;
\item $\xi(K_{r,s} \odot H) = r \cdot \n(H)  + s$;
\item $\xi(K_2(r,s)\odot H) = r \cdot \n(H)  + s$;
\item $\xi(K_{n_1,\ldots,n_p} \odot H) = \sum\limits_{i=1}^{p}n_i$;
\item $\xi(W_n \odot H) = n , \quad \text{if } n \ge 4 $;
\item $\xi(Q_n \odot H) = 2^{n-1}(\n(H) +1)$;
\item $\xi(P_{n}\odot H) =  \left\lfloor \dfrac{n}{2} \right\rfloor \n(H)+ \left\lceil \dfrac{n}{2} \right\rceil$;
\item $ \xi(C_n \odot H) = \left\lbrace
\begin{array}{cl}
  n, & \text{if } n \equiv 1 \Mod{2}
\\ \dfrac{n(\n(H) +1)}{2}, & \text{if } n \equiv 0 \Mod{2}
\end{array} \right.$
\end{enumerate}

\end{proposition}

\begin{proof}

Let us analyze each case separately.
\begin{enumerate}[{\rm (i)}]

\item The graph $K_2$ is bipartite with $\alpha(K_2)=1=\beta(K_2)$. Hence, by Theorem \ref{eqdimbipartG} we get $\xi(K_2 \odot H) =\n(H)+1$. Otherwise, if $n \ne 2$ we have $\beta(\widehat{K_n})=0$, and applying Theorem~\ref{totalEq} it follows that $\xi(K_n \odot H) =n$.

\item $K_{r,s}$ is a bipartite graph, with partite sets of cardinality $r,s$. Hence, $\xi(K_2(r,s)\odot H) = r \cdot \n(H)  + s$ as a direct consequence of Theorem~\ref{eqdimbipartG}.

\item Similarly to the previous case, $K_2(r,s)$ is a bipartite graph, with partite sets of cardinality $r,s$. Thus $\xi(K_2(r,s)\odot H) = r \cdot \n(H)  + s$.

\item For $p \geq 3$ every pair of vertices in $K_{n_1,\ldots,n_p}$ always have a common neighbour. Hence $\xi_t(K_{n_1,\ldots,n_p}) \leq +\infty$, which leads to $\xi(K_{n_1,\ldots,n_p} \odot H) = \n(K_{n_1,\ldots,n_p})$.

\item Since $W_n \cong K_1+C_{n-1}$, applying Proposition~\ref{eqdim_sumgraphs}, we have  $\xi(W_n \odot H) = \n(K_1)+\n(C_{n-1})=n$.

\item The hyper cube $Q_n$ is a bipartite graph, with partite sets of cardinality $ 2^{n-1}$. Thus the result follows from Theorem~\ref{eqdimbipartG}.

\item The path graph is a bipartite graph with partite sets whose cardinalities are $|A| = \left\lceil \dfrac{n}{2} \right\rceil$ and $|B| = \left\lfloor \dfrac{n}{2} \right\rfloor$. Therefore, by Theorem~\ref{eqdimbipartG}, the result follows.

\item If $n \equiv 0 \Mod{2}$, then $C_n$ is bipartite with partite sets of cardinality $ \dfrac{n}{2}$, and applying Theorem \ref{eqdimbipartG}, we get $\xi(C_n \odot H)= \dfrac{n(\n(H) +1)}{2}$. Otherwise, $\beta(\widehat{C_n})=0$, thus, by Theorem~\ref{totalEq}, we have $\xi(C_n \odot H) = n$.

\end{enumerate}

\end{proof}

\section{Threshold-dependent linearity of $\xi(G \odot H)$ respect to $\mathrm{n}(H)$} \label{Sec:linnh}

Based on the upper and lower bounds established in Theorem~\ref{improvedlbound} and Theorem~\ref{UpperBound1} respectively, it follows that, for any connected graph $G$,
$$
\lim_{\n(H) \to +\infty} \frac{\xi(G \odot H)}{\n(H)} = \beta(\widehat{G}).
$$
The above limit shows that $\xi(G \odot H)$ grows asymptotically linearly with $\n(H)$, with asymptotic growth rate given by $\beta(\widehat{G})$.

In this section, we go beyond this asymptotic characterization and prove a stronger result: for any connected graph $G$, there exists a threshold beyond which the equidistant dimension of $G \odot H$ is not only asymptotically linear, but in fact a linear function of $\n(H)$, with slope $\beta(\widehat{G})$.

We first establish a few preliminary results that will be instrumental in the analysis.

\begin{lemma} \label{Ubounds}
Let $G$ and $H$ be two graphs, with $G$ connected, and let $S$ be a $\xi(G \odot H)$-set satisfying the conditions of Lemma~\ref{structHindep}. 
If $\n(H) \ge 2$, then
\[
|U(S)| \le \beta(\widehat{G}) 
\;+\; 
\min\!\left\{
\frac{\alpha(\widehat{G})}{\n(H)},
\;
\frac{\beta(\widehat{G}) - \beta^*(G)}{\,\n(H)-1\,}
\right\}.
\]

\end{lemma}

\begin{proof}
By Lemma~\ref{vertxcovsets}, we have that $U(S)$ and $L(S)$ are a vertex covers of $\widehat{G}$. Applying the upper bound stated in Theorem~\ref{UpperBound1}, we obtain:
\begin{align*}
& |U(S)|\n(H) + |L(S)| = |S|  \le \beta(\widehat{G}) \n(H) + \n(G) ,\\
& \left(|U(S)| - \beta(\widehat{G})\right)\n(H) \le \n(G) -|L(S)| \le \n(G) - \beta(\widehat{G}) = \alpha(\widehat{G}).
\end{align*}
Which implies that $$|U(S)| \le \beta(\widehat{G}) +\dfrac{\, \alpha(\widehat{G})}{\n(H)}.$$ 
In the other hand, we can proceed as follows:
\begin{align*}
|S| &  = |U(S)|\n(H) + |L(S)| \\
& = \left(|U(S)| + \beta(\widehat{G}) - \beta(\widehat{G}) \right)\n(H) + |L(S)| \\
& = \beta(\widehat{G}) \n(H)+ \left(|U(S)| - \beta(\widehat{G}) \right)\n(H) + |L(S)|\\
& = \beta(\widehat{G})\n(H) + \left(|U(S)| - \beta(\widehat{G})\right)\n(H) + |L(S) \setminus U(S)| + |L(S) \cap U(S)| \\
& = \beta(\widehat{G})\n(H) + \left(|U(S)| - \beta(\widehat{G})\right)\n(H) + \n(\widehat{G})-|U(S)| + |L(S) \cap U(S)| \\
& = \beta(\widehat{G})\n(H) + \left(|U(S)| - \beta(\widehat{G})\right)\n(H) + \alpha(\widehat{G}) +\beta(\widehat{G}) -|U(S)| + |L(S) \cap U(S)| \\
& = \beta(\widehat{G})\n(H) + \left(|U(S)| - \beta(\widehat{G})\right) \left(\n(H)-1\right) + \alpha(\widehat{G})  + |L(S) \cap U(S)|
\end{align*}
Applying the upper bound of Theorem~\ref{UpperBound1} again, 
$$\beta(\widehat{G})\n(H) + \left(|U(S)| - \beta(\widehat{G})\right) \left(\n(H)-1\right) + \alpha(\widehat{G})  + |L(S) \cap U(S)| = |S| \le \beta(\widehat{G}) \n(H) + \n(G),$$
$$\left(|U(S)| - \beta(\widehat{G})\right) \left(\n(H)-1\right) + \alpha(\widehat{G})  + |L(S) \cap U(S)| \le  \n(G),$$
$$\left(|U(S)| - \beta(\widehat{G})\right) \left(\n(H)-1\right) \le  \n(G) - \alpha(\widehat{G})  - |L(S) \cap U(S)| = \beta(\widehat{G})- |L(S) \cap U(S)| \le \beta(\widehat{G})- \beta^*(G).$$
Leading to 
$$ |U(S)| \le \beta(\widehat{G}) +  \dfrac{\beta(\widehat{G})-\beta^*(G)}{\n(H)-1} .$$
Therefore, the result follows.
\end{proof}

\begin{proposition} \label{u_betaset_condition}
Let $G$ and $H$ be two graphs, with $G$ connected, and let $S$ be a $\xi(G \odot H)$-set that satisfies the conditions of Lemma~\ref{structHindep}. If $ \n(H) > \min\{ \alpha(\widehat{G}), \beta(\widehat{G})-\beta^*(G)+1\} $, then $U(S)$ is a $\beta(\widehat{G})$-set.
\end{proposition}

\begin{proof}
By Lemma~\ref{vertxcovsets} we have that $U(S)$ is a vertex cover of $\widehat{G}$. Thus $|U(S)| \ge \beta(\widehat{G})$. 

If $ \n(H) > \min\{ \alpha(\widehat{G}), \beta(\widehat{G})-\beta^*(G)+1\} $, then applying  Lemma \ref{Ubounds} we get:  
\begin{align*}
|U(S)| & \le \beta(\widehat{G}) + \min\left\lbrace\dfrac{\, \alpha(\widehat{G})}{\n(H)} , \dfrac{\beta(\widehat{G})-\beta^*(G)}{\n(H)-1} \,\right\rbrace \\
& \le \beta(\widehat{G}) + \min\left\lbrace\dfrac{\, \alpha(\widehat{G})}{\alpha(\widehat{G})+1} , \dfrac{\beta(\widehat{G})-\beta^*(G)}{\beta(\widehat{G})-\beta^*(G)+1} \,\right\rbrace\\
& \le \beta(\widehat{G}).
\end{align*}
Consequently, $|U(S)| = \beta(\widehat{G})$.
\end{proof}

For the remainder of this section, we define the auxiliary function
$$
R_G(H) = \xi(G \odot H) - \beta(\widehat{G}) \n(H),
$$
for any connected graph $G$ and any graph $H$.

\begin{lemma} \label{lemma_linear_Rg}
For any connected graph $G$ and any graph $H'$ such that $\n(H') = \min\{ \alpha(\widehat{G}), \beta(\widehat{G})-\beta^*(G)+1\}+1$, we have $R_G(H) \le R_G(H')$, for any graph $H$. In particular, if $\n(H) \ge \n(H')$, then $R_G(H)=R_G(H')$.
\end{lemma}

\begin{proof}

Let $S_H$ be a $\xi(G \odot H)$-set and $S_{H'}$ be a $\xi(G \odot H')$-set, both satisfying the conditions of Lemma~\ref{structHindep}.
Observe that
$$\xi(G \odot H) = |S_H| = |U(S_H)|\, \n(H) + |L(S_H)|,$$
and thus,
\begin{equation}\label{eq:rgh}
R_G(H) =  \left( |U(S_H)| - \beta(\widehat{G}) \right) \n(H) + |L(S_H)|. 
\end{equation}

Additionally, since $\n(H') > \min\{ \alpha(\widehat{G}), \beta(\widehat{G})-\beta^*(G)+1\}$, applying Proposition~\ref{u_betaset_condition} to the graph $H'$, it follows that \(|U(S_{H'})|=\beta(\widehat{G})\). Consequently,
$$\xi(G \odot H') = |S_{H'}| = \beta(\widehat{G})\, \n(H') + |L(S_{H'})|,$$
and hence,
\begin{equation}\label{eq:rgh'}
R_G(H') = |L(S_{H'})|. 
\end{equation}

Now, we consider the sets 
$$W_{H} = \left( \bigcup_{v_i \in U(S_H)} V(H_i') \right) \cup L(S_H), $$ and
$$W_{H'} = \left( \bigcup_{v_i \in U(S_{H'})} V(H_i) \right) \cup L(S_{H'}), $$
Applying Lemma~\ref{vertxcovsets} and Theorem~\ref{disteqGenralcond} to each of the previous sets, we conclude that $W_{H}$ is a distance-equalizer set of $G \odot H'$, and $W_{H'}$ is a distance-equalizer set of $G \odot H$. Thus, 
\begin{align}
 |U(S_H)|\, \n(H) + |L(S_H)| = |S_{H}| &\le |W_{H'}|= \beta(\widehat{G})\, \n(H) + |L(S_{H'})|,\label{W_H}\\
 \beta(\widehat{G})\, \n(H') + |L(S_{H'})|= |S_{H'}| &\le |W_{H}|= |U(S_H)|\, \n(H') +|L(S_H)|.\label{W_H'}
\end{align}
From \eqref{W_H} we obtain
\begin{equation}\label{R_H}
\left( |U(S_H)| - \beta(\widehat{G}) \right) \n(H) + |L(S_H)| \le |L(S_{H'})|,
\end{equation}
and from \eqref{W_H'} we have
\begin{equation}\label{R_H'}
|L(S_{H'})| \le \left( |U(S_H)| -  \beta(\widehat{G}) \right) \, \n(H') +|L(S_H)|. 
\end{equation}
Combining equations \eqref{eq:rgh} and  \eqref{eq:rgh'}, with inequalities \eqref{R_H} and \eqref{R_H'}, it follows that
\begin{equation}
R_G(H) \le R_G(H') \le \left( |U(S_H)| -  \beta(\widehat{G}) \right) \, \n(H') +|L(S_H)|.\label{ineq:RhleRh'}
\end{equation}

Therefore, $R_G(H) \le R_G(H')$ for any graph $H$. In particular, if $\n(H) \ge \n(H') = \min\{ \alpha(\widehat{G}), \beta(\widehat{G})-\beta^*(G)+1\}+1$, then by applying Proposition~\ref{u_betaset_condition} to the graph $H$, we get \(|U(S_{H})|=\beta(\widehat{G})\). Thus, according to equation \eqref{eq:rgh}, $R_G(H)=|L(S_H)|$. Therefore, from \eqref{ineq:RhleRh'}, we obtain
$$ R_G(H) = R_G(H'), $$
for every graph $H$ such that $ \n(H) >  \min\{ \alpha(\widehat{G}), \beta(\widehat{G})-\beta^*(G)+1\}$.
\end{proof}

\begin{theorem} \label{linearity_nh}

For any connected graph $G$ there exists a number $k(G) \in \mathbb{N}$, with $\alpha(\widehat{G}) + \beta^{*}(\widehat{G}) \le k(G) \le \n(G)$, such that
$$\xi(G \odot H) = \beta(\widehat{G})\, \n(H) + k(G),$$ for any graph $H$ satisfying $\n(H) > \min\{ \alpha(\widehat{G}), \beta(\widehat{G})-\beta^*(G)+1\},$
and $$\xi(G \odot H) \le \beta(\widehat{G})\, \n(H) + k(G),$$ for any graph $H$ satisfying $\n(H) \le \min\{ \alpha(\widehat{G}), \beta(\widehat{G})-\beta^*(G)+1\}.$

\end{theorem} 

\begin{proof}
Let $H'$ a graph such that $\n(H') = \min\{ \alpha(\widehat{G}), \beta(\widehat{G})-\beta^*(G)+1\} +1$, and $k(G)=R_G(H')$. Applying the upper and lower bounds for $\xi(G \odot H')$, stated in Theorem~\ref{UpperBound1} and Theorem~\ref{improvedlbound}, we get 
$$\alpha(\widehat{G}) + \beta^*(G) \le k(G) \le \n(G).$$ 
Now, since $\xi(G \odot H)= \beta(\widehat{G})\, \n(H) + R_G(H)$, the result follows directly from Lemma~\ref{lemma_linear_Rg}.
\end{proof}

The example provided in Figure \ref{fig:FigFish} can be used to clarify that not all values of $\xi(G \odot H)$ lie on the same straight line. 
Let $G$ be the graph of order 6 shown in Figure \ref{fig:FigFish}, and $H$ any graph. For any $\xi(G \odot H)$-set $S$, which satisfies the conditions of Lemma~\ref{structHindep}, we have that $$|U(S)| \le \beta(\widehat{G}) + \dfrac{\beta(\widehat{G})-\beta^*(G)}{\n(H)-1} \le 2 \quad \text{for } \n(H) \ge 2,$$ as a consequence of Lemma \ref{Ubounds}.

Let us denote $U_1 = \{ 4 \}$, $L_1 = V(G)$. Observe that, in this case, $U_1$ is the only $\beta(\widehat{G})$-set. Thus, we have $\beta(\widehat{G})=1$ and $\alpha(\widehat{G})=5$. Moreover, $(U_1, L_1)$ is a forward-equalized pair of $G$. Note that $(U_1, L_1 \setminus \{4\})$ is not a forward-equalized pair of $G$, since $d_G(w,4) \ne d_G(w,3)+1$ for every $w \in V(G)$. Hence, by Theorem~\ref{disteqGenralcond}, the unique possible $\xi(G \odot H)$-set $S$ with $|U(S)| = 1$ is
\[
S_1(H) = \left( \bigcup_{v_i \in U_1} V(H_i) \right) \cup L_1 .
\]
 
For the case when $|U(S)| = 2$, the sets $U_2 = \{ 3, 4 \}$, $L_2 = \{ 1, 2, 5, 6 \}$ form a forward-equalized pair of $G$.  
Now, since $|L(S)| \ge \n(G) - |U(S)| = 4$, the set
\[
S_2(H) = \left( \bigcup_{v_i \in U_2} V(H_i) \right) \cup L_2
\]
has minimum cardinality among all distance-equalizer sets $S$ of $G \odot H$ such that $|U(S)| = 2$ and $\n(H) \ge 2$.
Consequently, $$\xi(G \odot H) =\min\{|S_1(H)|, |S_2(H)|\}= \min\{\n(H)+6, 2\n(H)+4\},\quad \text{for } \n(H) \ge 2.$$
Now, observe that for $\n(H)=1$ we have $|S_2(H)|=6=\n(G)<|S_1(H)|$, thus the lower bound stated in Lemma \ref{lemmaLowerBound} is attained by $|S_2(H)|$. Nevertheless, for $\n(H) \ge 2$, we have $|S_1(H)| \le |S_2(H)|$. Therefore,

\begin{equation}\label{formulaFish}
\xi(G \odot H) = \left\lbrace
\begin{array}{cl}
  6, & \text{if } \n(H) = 1
\\ \n(H) + 6, & \text{if }  \n(H) \ge 2
\end{array} \right.
\end{equation}

\begin{figure}[h!]
  \centering
  \begin{subfigure}{0.45\textwidth}
    \centering
    \begin{tikzpicture}[transform shape, inner sep = 1mm]
	\node[ draw, circle, fill=black, label=above:3] (v3) at (0,0) {};
	\node[ draw, circle, fill=black, label=above:6] (v6) at (2,2) {};
	\node[draw, circle, fill=black, label=below:1] (v1) at (-2,-2) {};
	\node[ draw, circle, fill=black, label=above:4] (v4) at (-4,0) {};
	\node[ draw, circle, fill=black, label=below:5] (v5) at (2,-2) {};
	\node[ draw, circle, fill=black, label=above:2] (v2) at (-2,2) {};
	\draw (v1)--(v3)--(v2)--(v4)--(v1)--(v2);
	\draw (v3)--(v5)--(v6)--(v3);

    \end{tikzpicture}
    
  \end{subfigure}
  \hfill
  \begin{subfigure}{0.45\textwidth}
    \centering
    \begin{tikzpicture}[transform shape, inner sep = 1mm]
	\node[ draw, circle, fill=gray, label=left:4] (v4) at (-3,0) {};
	\node[ draw, circle, fill=gray, label=right:1] (v1) at (0.5,2) {};
	\node[draw, circle, fill=gray, label=right:2] (v2) at (1,1) {};
	\node[ draw, circle, fill=gray, label=right:3] (v3) at (1.5,0) {};
	\node[ draw, circle, fill=gray, label=right:5] (v5) at (1,-1) {};
	\node[ draw, circle, fill=gray, label=right:6] (v6) at (0.5,-2) {};
	\draw (v4) -- (v6);
	\draw (v4) -- (v5);
    \end{tikzpicture}
   
  \end{subfigure}
\caption{A Graph $G$ (to the left), and its corresponding empty bisector graph $\widehat{G}$ (to the right), such that $\xi(G \odot H) = \beta(\widehat{G})\n(H)+ \n(G)$ for every graph $H$ with $\n(H) \ge 2$, but $\xi(G \odot H) < \beta(\widehat{G})\n(H) + \n(G)$ whenever $\n(H) = 1$.}
  \label{fig:FigFish}
\end{figure}
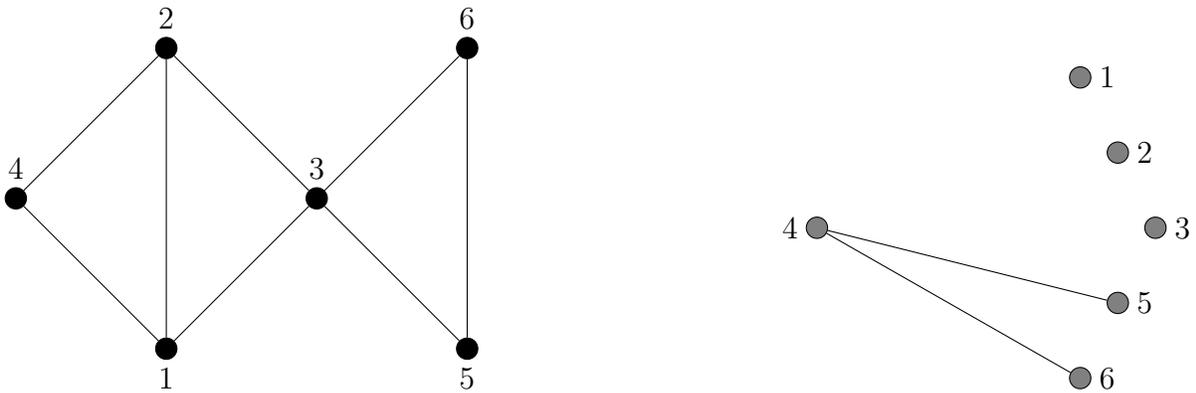

In Table \ref{tableFish}, we present the computed values of the general lower and upper bounds established in Theorems~\ref{improvedlbound} and \ref{UpperBound1}, together with the cardinalities of $S_1$ and $S_2$, for $\n(H) \le 3$. Observe that, in this case, the lower bound $\n(H)+5$ is attained for $\n(H)=1$, while the upper bound $\n(H)+6$ is attained for $\n(H) \ge 2$. Finally, observe that equation~\eqref{formulaFish} shows that the values of $\xi(G \odot H)$ lie on the same straight line, except in the case $\n(H)=1$.
 
\begin{table}[h]
\centering   
\begin{tabular}{|c|c|c|c|c|c|} 
\hline 
$\n(H)$ & $\n(H)+5$ & $|S_1|$ & $|S_2|$ & $\n(H)+6$ & $\xi(G \odot H)$\\ 
\hline 
1 & 6 & 7 & 6 & 7 & 6\\ 
\hline 
2 & 7 & 8 & 8 & 8 & 8\\ 
\hline
3 & 8 & 9 & 10 & 9 & 9\\ 
\hline
\end{tabular}
\caption{Calculation of upper and lower bounds for $\xi(G \odot H)$, where $G$ is the graph shown in Figure~\ref{fig:FigFish}. } \label{tableFish}
\end{table}

Another interesting example, depicted in Figure~\ref{fig:NonLinearity}, illustrates that more than one value of \( \xi(G \odot H) \) may fall below the line of slope \( \beta(\widehat{G}) \) given in Theorem~\ref{linearity_nh} when
$ \n(H) < \min\{ \alpha(\widehat{G}),\, \beta(\widehat{G}) - \beta^{*}(G) + 1 \} $.

Let $G$ be a graph of order $8$ and let $\widehat{G}$ be its corresponding empty bisector graph,
both shown in Figure \ref{fig:NonLinearity}, and let $H$ be an arbitrary graph. Applying Lemma~\ref{Ubounds}, for any $\xi(G \odot H)$-set $S$, which satisfies the conditions of Lemma~\ref{structHindep}, we have  
\begin{equation} \label{eq:nonlin}
|U(S)|\; \le \; 3 \;+\; \min\!\left\{\frac{5}{\n(H)},\;\frac{3}{\,\n(H)-1\,}\right\} \; \le \; 5 \quad \text{for } \n(H) \ge 2.
\end{equation} 

Analogously to the previous example we will analyse all values of $|S|$ for each possible case of $|U(S)|$ and for any value of $\n(H)$. It is worth keeping in mind that, by Lemma \ref{vertxcovsets}
and Theorem \ref{disteqGenralcond},
$U(S), L(S) \in \mathcal{B}(\widehat{G})$, $U(S) \cup L(S) = V(G)$,
and $(U(S), L(S))$ is a forward-equalized pair of $G$.

If $ \n(H) > \min\{ \alpha(\widehat{G}),\, \beta(\widehat{G}) - \beta^{*}(G) + 1 \} = 4 $, then by Proposition~\ref{u_betaset_condition}, $|U(S)|=\beta(\widehat{G})=3$. Observe that, in the graph $G$ shown in Figure~\ref{fig:NonLinearity} the only possible $\beta(\widehat{G})$-sets are  $U_1 = \{ 1,2,3 \}$ and $U_2 = \{ 2,3,4 \}$. Thus, $U(S)= U_1$ or $U(S)= U_2$. 
We initially assume that $U(S) = U_1$. Since for every vertex $u \in U_1$, every vertex $v \in V(G)\setminus U_1$, and every vertex $w \in V(G)$ it holds that $d_G(w,u) \neq d_G(w,v)+1$, it is necessary that $U_1 \subseteq L(S)$ for $(U_1, L(S))$ to be a forward-equalized pair.
Hence, since $V(G) \setminus U_1 \subseteq L(S)$, the only possible set $L(S)$ such that $(U_1, L(S))$ forms a forward-equalized pair of $G$ is
$L(S)=V(G)$. This leads to $|S| = \beta(\widehat{G})\n(H)+\n(G)=3\n(H)+8$. Now suppose that $U(S) = U_2$.
Since for every vertex $u \in \{2,3\}$ and every vertex $w \in V(G)$ there exists a vertex $v \in V(G) \setminus U_2$ such that $d_G(w,u) \neq d_G(w,v)+1$, it follows that $\{2,3\} \subseteq L(S)$.
On the other hand, for any vertex $v \in V(G) \setminus U_2$, there exists $w \in V(G)$ such that
$d_G(w,4) = d_G(w,v)+1$.
Taking into account the minimality assumption on $|S|$, this implies that
$L(S) = V(G) \setminus \{4\}$.
Therefore, in the case $U(S) = U_2$, we obtain the minimum possible value of
$|S|$ when $U(S)$ is a $\beta(\widehat{G})$-set. Consequently,
\begin{equation}\label{U_S_3}
\xi(G \odot H) = \beta(\widehat{G})\,\n(H) + \n(G) - 1 = 3\n(H) + 7,
\end{equation}
for $\n(H) > 4$.

Note that, by \eqref{eq:nonlin}, $|U(S)| \in \{3,4\}$ whenever $\n(H) \in \{3,4\}$.
Moreover, since $\big(\{3,4,5,6\}, \{1,2,7,8\}\big)$ is a forward-equalized pair of $G$, there exists a distance equalizer set of minimum cardinality $4\n(H)+4$
among all distance equalizer sets $S'$ of $G \odot H$ with $|U(S')| = 4$. Comparing this cardinality with the minimum cardinality obtained in
\eqref{U_S_3} for $|U(S)|=3$, we obtain
\[
\min_{\n(H) \in \{3,4\}} \left\{ 4\n(H)+4,\; 3\n(H)+7 \right\} = 3\n(H)+7.
\]
Together with equation~\eqref{U_S_3}, this yields $\xi(G \odot H) = 3\n(H)+7$ for $\n(H)\ge 3$.

If $\n(H)=2$, then \eqref{eq:nonlin} implies that $|U(S)| \in \{3,4,5\}$. Furthermore, the existence of the forward-equalized pair
$\big(\{4,5,6,7,8\}, \{1,2,3\}\big)$ in $G$ guarantees the existence of
a distance equalizer set of minimum cardinality $5\n(H)+3$
among all distance equalizer sets $S'$ of $G \odot H$ with $|U(S')| = 5$. 
Therefore, taking into account all the arguments presented above, we deduce that
if $\n(H) = 2$, then
\[
\xi(G \odot H)
= \min_{\n(H) = 2} \left\{ 5\n(H)+3,\; 4\n(H)+4,\; 3\n(H)+7 \right\}=4\n(H)+4
= 12.
\]
Finally, since $(U_3,L_3)=\big(\{3,4,5,6\}, \{1,2,7,8\}\big)$ is a disjoint forward-equalized pair of $G$ with $U_3,L_3\in\mathcal{B}(\widehat{G})$, we deduce that $\beta^*(G)=0$.
Therefore, when $\n(H)=1$, by applying Theorem~\ref{totalEq}, we obtain
$\xi(G \odot H) = \n(G) = 8$.

The values obtained for the equidistant dimension of $G \odot H$, where $G$ is
the graph depicted in Figure~\ref{fig:NonLinearity}, can be summarized by the following formula:

$$\xi(G \odot H) = \left\lbrace
\begin{array}{cl}
  4 \n(H)+4, & \text{if } \n(H) \le 2
\\ 3\n(H)+7, & \text{if }  \n(H) > 2
\end{array} \right.$$

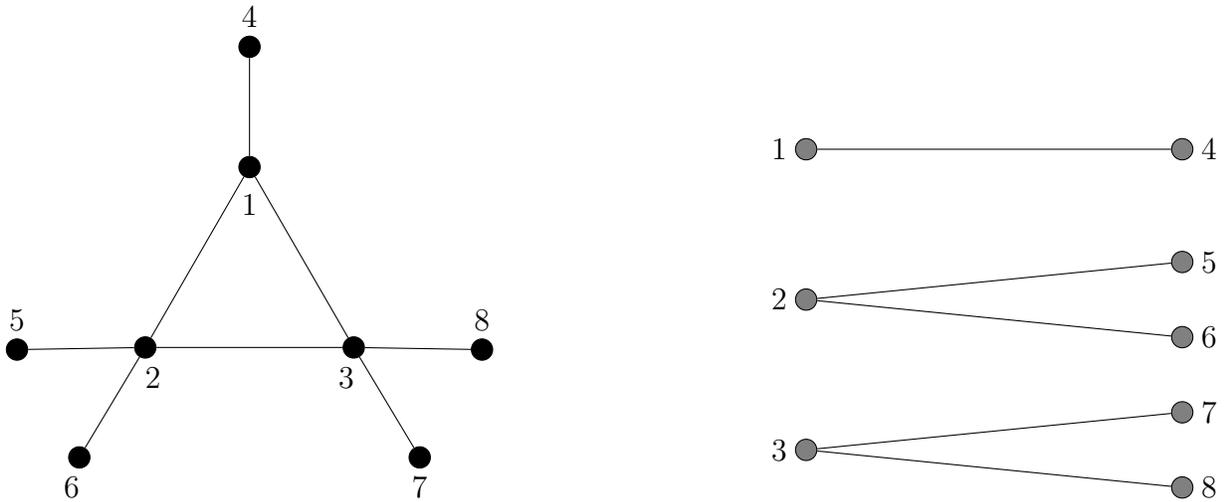
\begin{figure}[h!]
  \centering
  \begin{subfigure}{0.45\textwidth}
    \centering
	\begin{tikzpicture}[transform shape, inner sep = 1mm]
\def\radius{1.6} 
\foreach \ind in {1,2,3}
{
\pgfmathsetmacro{\a}{-30+(360*\ind)/3};
\pgfmathparse{360/3*\ind};
\node [draw=black, shape=circle, fill=black] (v\ind) at (\a:\radius cm) {};
}

\node [draw=black, shape=circle, fill=black] (v4) at (90:3.2 cm) {};

\draw[black] (v1)--(v2)--(v3)--(v1)--(v4);

\def\radius{3.2} 
\foreach \ind in {5,6}
{
\pgfmathsetmacro{\a}{-15+(360*\ind)/3};
\pgfmathparse{360/3*\ind};
\node [draw=black, shape=circle, fill=black] (v\ind) at (\a:\radius cm) {};
}

\def\radius{3.2} 
\foreach \ind in {8,9}
{
\pgfmathsetmacro{\a}{-45+(360*\ind)/3};
\pgfmathparse{360/3*\ind};
\node [draw=black, shape=circle, fill=black] (v\ind) at (\a:\radius cm) {};
}

\draw[black] (v5)--(v2)--(v8);
\draw[black] (v6)--(v3)--(v9);

\node at ([xshift=0 cm, yshift=-.5 cm]v1) {1};
\node at ([xshift=0.1 cm, yshift=-.4 cm]v2) {2};
\node  at ([xshift=-0.1 cm, yshift=-.4 cm]v3) {3};
\node  at ([yshift=0.4 cm]v4) {4};
\node  at ([yshift=.4 cm]v8) {5};
\node  at ([xshift=-.1 cm, yshift=-.4 cm]v5) {6};
\node  at ([yshift=-.4 cm]v9) {7};
\node at ([yshift=.4 cm]v6) {8};

\end{tikzpicture}

  \end{subfigure}
  \hfill
  \begin{subfigure}{0.45\textwidth}
    \centering
    \begin{tikzpicture}[transform shape, inner sep = 1mm]

\node[ draw, circle, fill=gray, label=left:1] (v1) at (-2.5,1) {};
\node[ draw, circle, fill=gray, label=left:2] (v2) at (-2.5,-1) {};
\node[ draw, circle, fill=gray, label=left:3] (v3) at (-2.5,-3) {};
\node[ draw, circle, fill=gray, label=right:4] (v4) at (2.5,1) {};
\node[ draw, circle, fill=gray, label=right:5] (v5) at (2.5,-0.5) {};
\node[ draw, circle, fill=gray, label=right:6] (v6) at (2.5,-1.5) {};
\node[ draw, circle, fill=gray, label=right:7] (v7) at (2.5,-2.5) {};
\node[ draw, circle, fill=gray, label=right:8] (v8) at (2.5,-3.5) {};

\draw (v1)--(v4);
\draw (v5)-- (v2)--(v6);
\draw (v7)-- (v3)--(v8);

\end{tikzpicture}

  \end{subfigure}
 \caption{Graph $G$ (to the left), and its corresponding empty bisector graph $\widehat{G}$ (to the right), with  $\beta(\widehat{G}) = 3$, $ \alpha(\widehat{G})=5$ and $\beta^*(G) = 0$.}\label{fig:NonLinearity}
\end{figure}

Figure~\ref{graphic:noLinear} shows that, for $\n(H)=1$ and $\n(H)=2$, the equidistant
dimension of $G \odot H$ falls below the line corresponding to the values of $\xi(G \odot H)$ for $\n(H) > 2$. Note that, in this case, the lower bound $3\n(H)+5$ given in Theorem~\ref{improvedlbound} is attained only for $\n(H)=1$, whereas the upper bound $3\n(H)+8$ from Theorem~\ref{UpperBound1} is never attained for any value of $\n(H)$.

\begin{figure}[htbp]
\centering
\begin{tikzpicture}
\begin{axis}[
    axis lines=middle,
    xmin=0.5, xmax=7,
    ymin=4, ymax=25,
    xlabel style={xshift=35pt, yshift=-10pt},
	ylabel style={yshift=5pt},
    xlabel={$\n(H)$},
    ylabel={$\xi(G \odot H)$},
    width=12cm,
    height=8cm,
    enlarge x limits=0.1,
    enlarge y limits=0.1,
    tick align=outside,
    ticklabel style={font=\small},
    label style={font=\small},
    grid=none,
    legend style={
        at={(1.1,0.8)},  
        anchor=north,
        draw=none,
        fill=none,
        font=\small,
        legend columns=0,
    },
    legend cell align=left,
    legend image code/.code={}, 
    legend entries={
        \shortstack[l]{%
	    $\displaystyle \bullet \, \xi(G \odot H) =
	    \left\lbrace
	    \begin{array}{cl}
	    4\n(H)+4, & \text{if } \n(H) \le 2
	    \\ 3\n(H) + 7, & \text{if }  \n(H) > 2
	    \end{array} \right.$
        }
    },
]

\addplot[
    only marks,
    mark=*,
    black,
    nodes near coords,            
    every node near coord/.append style={
        font=\footnotesize,
        anchor=north,
        yshift=-2pt
    },
]
coordinates {
    (1,8) (2,12) (3,16) (4,19) (5,22) (6,25)
};

\addplot[
    black,
    dashed,
    domain=0:7,
    samples=100,
]
{3*x + 7};

\end{axis}
\end{tikzpicture}
\caption{The graphical representation of $\xi(G \odot H)$, where $G$ is the graph shown in Figure~\ref{fig:NonLinearity}.}
\label{graphic:noLinear}
\end{figure}
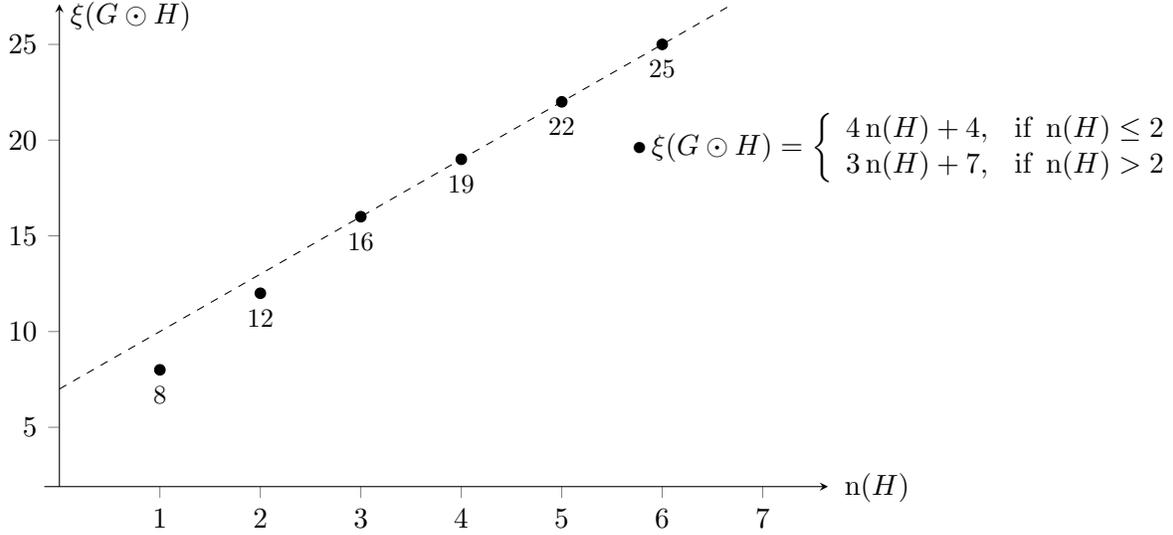

\section{Particular cases and minor results} \label{Sec:MinorResults}

In this section we establish some immediate relationships between the properties of $\widehat{G}$ and the properties of $G$. We also analyse several cases in which specific conditions are imposed on the base graph $G$.

\begin{proposition} \label{RelationsGprime}
For every connected graph $G$,

\begin{enumerate}[{\rm(i)}]
 \item $\xi(G) \geq \beta(\widehat{G})$;
  \item $\Delta(G) \leq \alpha(\widehat{G})$;
 \item $\omega(G) \leq \alpha(\widehat{G})$, for $\n(G) \ne 2$.

\end{enumerate}

\end{proposition}

\begin{proof}

We will analyse each case separately.

\begin{enumerate}[{\rm(i)}]

\item Let $S$ be a $\xi(G)$-set. Then for every $u,v \in V(G)\setminus S$, we have that $B_G(u \mid v) \neq \varnothing $. Hence, $V(G)\setminus S$ is an independent set in $\widehat{G}$, and $S$ is vertex cover of $\widehat{G}$. As a consequence $\xi(G) \geq \beta(\widehat{G})$.

\item Let $u \in V(G)$ a vertex with $\delta(u) = \Delta(G)$, then $N(u)$ is an independent set in $\widehat{G}$. Thus $\Delta(G) \leq \alpha(\widehat{G})$.

\item For $\n(G) \ne 2$ every clique in $G$ is an independent set in $\widehat{G}$. Thus, $\omega(G) \leq \alpha(\widehat{G})$.
\end{enumerate}
\end{proof}

It was proved in \cite{Gonzalez2022} that for any positive integers $r$ and $s$ with $s \geq r$, we have $\xi(K_{r,s}) = r$.  From this result and Theorem \ref{eqdimbipartG}, we deduce that the bound given in Proposition \ref{RelationsGprime}(i) is reached when $G \cong K_{r,s}$. Moreover, the equalities in Proposition \ref{RelationsGprime} (i), (ii) and (iii) are achieved for $ G \cong K_n \odot N_1 $. In order to illustrate this fact, we provide an explicit example in Figure \ref{fig:Fig3}, for the case $n=5$.

\begin{figure}[h!]
  \centering
  \begin{subfigure}{0.45\textwidth}
    \centering
	\begin{tikzpicture}[transform shape, inner sep = 1mm]
\def\radius{1.5} 
\foreach \ind in {0,...,4}
{
\pgfmathsetmacro{\a}{18+(360*\ind)/5};
\pgfmathparse{360/5*\ind};
\node [draw=black, shape=circle, fill=black] (v\ind) at (\a:\radius cm) {};
}

\foreach \ind in {0,...,4}
{
\pgfmathsetmacro{\j}{mod(\ind + 1, 5)};
\pgfmathparse{int(\j)};
\draw[black] (v\ind) -- (v\pgfmathresult);
\pgfmathsetmacro{\k}{mod(\ind + 2, 5)};
\pgfmathparse{int(\k)};
\draw[black] (v\ind) -- (v\pgfmathresult);
}

\def\radius{3} 
\foreach \ind in {5,...,9}
{
\pgfmathparse{18+360/5*\ind};
\node [draw=black, shape=circle, fill=black] (v\ind) at (\pgfmathresult:\radius cm) {};
}

\foreach \ind in {0,...,4}
{
\pgfmathparse{int(\ind+5)}
\draw[black] (v\ind) -- (v\pgfmathresult);
}

\node at ([xshift=.1 cm, yshift=-.4 cm]v0) {1};
\node at ([xshift=-.2 cm, yshift=.4 cm]v1) {2};
\node  at ([xshift=-.2 cm, yshift=-.4 cm]v2) {3};
\node  at ([yshift=-.4 cm]v3) {4};
\node  at ([xshift=-.1 cm, yshift=-.4 cm]v4) {5};
\node  at ([yshift=.4 cm]v5) {6};
\node at ([yshift=.4 cm]v6) {7};
\node  at ([yshift=-.4 cm]v7) {8};
\node  at ([xshift=-.1cm, yshift=-.4 cm]v8) {9};
\node  at ([yshift=-.4 cm]v9) {10};

\end{tikzpicture}

  \end{subfigure}
  \hfill
  \begin{subfigure}{0.45\textwidth}
    \centering
    \begin{tikzpicture}[transform shape, inner sep = 1mm]
    
\foreach \ind in {1,...,5}
{
\pgfmathparse{int(\ind + 5)};
\node[ draw, circle, fill=gray, label=left:\ind] (v\ind) at (-2.5,-1.5*\ind) {};
\node[ draw, circle, fill=gray, label=right:\pgfmathresult] (v\pgfmathresult) at (2.5,-1.5*\ind) {};
\draw (v\ind) -- (v\pgfmathresult);
 }
 
\end{tikzpicture}

  \end{subfigure}
\caption{Graph $G \cong K_5 \odot N_1$ (to the left), and its corresponding empty bisector graph $\widehat{G} \cong N_5 \odot N_1$ (to the right), with  $\beta(\widehat{G}) = \alpha(\widehat{G}) = \xi(G) = \Delta(G) = \omega(G) = 5$ and $\beta^*(G) = 0$.}
   \label{fig:Fig3}
\end{figure}
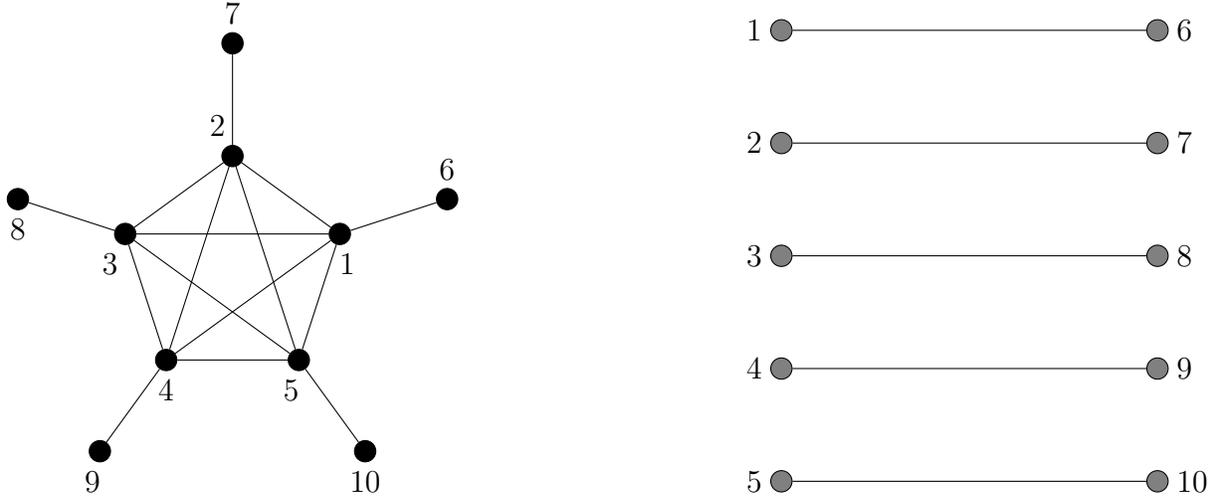

From Proposition \ref{RelationsGprime} (i) and Theorem~\ref{UpperBound1}, we can establish a relationship between $\xi (G \odot H)$ and $ \xi(G)$.

\begin{proposition}\label{xiupperbound}
For every connected graph $G$ and any graph $H$, $$\xi (G \odot H) \leq \xi(G)\n(H) + \n(G).$$
\end{proposition}

The upper bound given in Proposition~\ref{xiupperbound} is tight.
The equality
$$
\xi (G \odot H) = \xi(G)\,\n(H) + \n(G)
$$
is attained, in particular, when $G \cong K_n \odot N_m$, with $n \ge 3$, $m \ge 2$ and $\n(H) > n+1$.
For this case, the associated empty bisector graph $\widehat{G}$ is isomorphic to the bipartite graph $N_n \odot N_m$. Since 
$V(N_n)$ is the unique $\beta(N_n \odot N_m)$-set, and by Proposition~\ref{familieseqidim}~(i), $\xi(K_n \odot N_m)= n$, we get
$$
\xi(G) = \beta(\widehat{G}).
$$
Note that, $\big(V(K_n \odot N_m) \setminus V(K_n) , V(K_n) \big)$ constitute a disjoint forward-equalized pair of $G$, thus $\beta^*(G)=0$. 
Let $S$ be a $\xi(G \odot H)$-set satisfying the structural conditions of Lemma~\ref{structHindep}. If
$\n(H) > n+1 \ge \beta(\widehat{G}) - \beta^*(G) + 1$,
then by Proposition~\ref{u_betaset_condition}, the set $U(S)$ is necessarily a $\beta(\widehat{G})$-set. Recalling that $V(G)=V(\widehat{G})$, we have
$U(S)=V(K_n)$.
Let us denote by $\{v_i\}_{i=1}^n$ the vertices of $K_n$, and let $N_m^i$ be the copy of $N_m$ attached to $v_i$ in $K_n \odot N_m$.
For any vertex $v_i \in U(S)$ and any vertex $h_j \in V \bigl( N_m^j \bigr) \subseteq L(S)$,
there exists no vertex $w \in V(G)$ such that
$$
d_{G}(w,v_i)= d_{G}(w,h_j)+1.
$$
Therefore, $L(S) = V(G) = V(K_n \odot N_m)$ is the unique subset $L(S) \in \mathcal{B}(\widehat{G})$ such that $(U(S),L(S))$ is a forward-equalized pair of $G$. Hence, the previous reasoning together with Theorem~\ref{disteqGenralcond} leads to 

$$
\begin{aligned}
\xi(G \odot H)=|S|
&= |U(S)|\,\n(H) + |L(S)| \\
&= \beta(\widehat{G})\,\n(H) + \n(G) \\
&= \xi(G)\,\n(H) + \n(G),
\end{aligned}
$$
for all $\n(H) > n+1$.

The previous example illustrates that the parameter $\beta^*(G)$ is not always attained as the intersection $|U \cap L|$ of a forward-equalized pair $(U,L)$, where $U$ is a $\beta(\widehat{G})$-set. 

\begin{proposition} \label{excentr2bipart}
Let $G$ and $H$ be a two graphs. If there exists a vertex $u \in V(G)$ with $\varepsilon(u) \le 2$, then
\begin{enumerate}[{\rm(i)}]
 \item  the sets $ L = N_G(u) $ and $ U = V(G) \setminus N_G(u) $ form a bipartition of $\widehat{G}$,
 \item the set
$ S =  \left( \bigcup\limits_{v_i \in U} V(H_i)\right)\cup L$ 
is a distance-equalizer set of $G \odot H$,
 \item $\beta^*(G)=0$,
 \item $\xi(G \odot K_1) = \n(G)$.
\end{enumerate}

\end{proposition}

\begin{proof}
We will analyse each case separately.

\begin{enumerate}[{\rm(i)}]

\item
Note that $L$ and $U$ are disjoint sets and $L \cup U = V(\widehat{G})$.
For any two vertices $v, w \in L$, we have $u \in B_G(v \mid w)$, which implies that $L$ is an independent set in $\widehat{G}$.
Now, in order to prove that $U$ is also an independent set of $\widehat{G}$, we distinguish two possible cases for any two distinct vertices $v, w \in U$:

\begin{itemize}
\item $v, w \in U \setminus \{u\}$.  
In this case $d_G(v,u)=d_G(u,w)=2$, thus $u \in B_G(v \mid w)$.

\item $v = u$ and $w \in U \setminus \{u\}$.  
In this case $d_G(u,w)=2$, hence there exists a vertex $x \in N_G(u)=L$ such that
$d_G(v,x)=1=d_G(x,w)$, leading to $B_G(v \mid w) \neq \varnothing$.
\end{itemize}

Consequently, $L$ and $U$ are partite sets of the bipartite graph $\widehat{G}$.

\item
Since $U$ and $L$ are independent sets of $\widehat{G}$, we have $V(\widehat{G})\setminus U = L$ and $V(\widehat{G})\setminus L = U$ are vertex covers of $\widehat{G}$. Let $x \in L$ and $y \in U$. If $x \in N_G(y)$, then $d_G(x,y)=d_G(x,x)+1$. Otherwise, if $x \notin N_G(y)$, then $d_G(u,y)=d_G(u,x)+1=2$. Hence, $(U, L)$ is a forward-equalized pair of $G$. Now, applying  Theorem~\ref{disteqGenralcond}, it follows that $S$ is a distance-equalizer set of $G \odot H$.

\item 
The sets $U$ and $L$ are vertex covers of $\widehat{G}$, and $(U,L)$ forms a forward-equalized pair of $G$. Since $U \cap L = \varnothing$, it follows that $\beta^*(G)=0$.
 
\item 
As $\beta^*(G)=0$, by Theorem~\ref{totalEq}, we have that $\xi(G \odot K_1) = \n(G)$.

\end{enumerate}

\end{proof}

\begin{corollary} \label{excent2grmax}
Let $G$ and $H$ be a two graphs. If there exists a vertex $u\in V(G)$ with $ \varepsilon(u) \le 2 $, then the graph $ \widehat{G} $ is bipartite and 
 $$
 \xi(G \odot H) \le \bigl(\n(G) - \delta(u)\bigr)\n(H) + \delta(u).
$$
\end{corollary}

Relying on Corollary \ref{excent2grmax} we can provide a family of base graphs $G$ for which the lower bound stated in Corollary \ref{lowerbipartite} is attained.

\begin{proposition} \label{excent2_xi_eq}
Let $G$ and $H$ be a two graphs. If there exists a vertex $u\in V(G)$ with $ \varepsilon(u) \le 2$ and $\delta(u) = \alpha(\widehat{G})$, then
$$ \xi(G \odot H) = \beta(\widehat{G})\,\n(H) + \alpha(\widehat{G}) $$
\end{proposition} 

\begin{proof}
Applying Corollary \ref{excent2grmax} and the Gallai's Theorem \ref{Gallai}, we have 

$$ \xi(G \odot H) \le \bigl(\n(G) - \delta(u)\bigr)\n(H) + \delta(u) = \beta(\widehat{G})\,\n(H) + \alpha(\widehat{G}).$$
Now, applying Corollary \ref{lowerbipartite} the equality follows.

\end{proof}

\begin{proposition}
For every graph $G$, such that $\Delta(G)= \n(G)-1$, and every graph $H$,

\[ \xi(G \odot H) = \left\lbrace
\begin{array}{cl}
  \n(G), & \text{if } \delta(G) \geq 2 \\
  \n(H) + \n(G)-1, & \text{if } \delta(G)=1
\end{array} \right. \]

\end{proposition}

\begin{proof} Let be $u$ an universal vertex of $G$. Note that $N_G(u)$ is an independent set in $\widehat{G}$.

If $\delta(G) \geq 2$, then there exists a vertex $w \in N_G(u) \cap N_G(v)\subseteq B_G(v \mid u)$ for every $v \in V(G) \setminus \{u\}$. Consequently, $\alpha(\widehat{G})= \n(G)$, which is equivalent to $\beta(\widehat{G})=0$. Now, applying Theorem~\ref{totalEq}, we get $\xi(G \odot H)= \n(G)$.

If $\delta(G)=1$, then for every vertex $v\in V(G)$ of degree one we have $B_G(v \mid u)=\varnothing$. Consequently, $\{u,v\} \in E(\widehat{G})$.
It follows that $N_G(u)$ is an $\alpha(\widehat{G})$-set. Moreover, in this case we have
$\alpha(\widehat{G}) = \n(G)-1 = \delta(u)$.
Hence, by Proposition~\ref{excent2_xi_eq}, it follows that
\[
\xi(G \odot H)
= \bigl(\n(G)-\delta(u)\bigr)\n(H) + \delta(u)
= \n(H) + \n(G) - 1.
\]
This completes the proof.
\end{proof}

\section{Concluding remarks and future work} \label{Sec:Conclusions}

In this paper, we have undertaken a systematic study of the equidistant dimension of the corona product $G \odot H$ of two graphs $G$ and $H$. To this end, we introduced the \emph{empty bisector graph} $\widehat{G}$, a novel auxiliary construction that plays a central role in our analysis. This graph encodes pairs of vertices in $G$ that cannot be equidistantly distinguished by any third vertex, thereby providing an effective framework for capturing the equidistance constraints inherent to the base graph. Using this approach, we derived tight lower and upper bounds for the equidistant dimension of general corona product graphs and determined the exact value of $\xi(G \odot H)$ for several classical families of graphs.

The empty bisector graph emerges as a central tool in the study of the equidistant dimension of corona products. Although we established several relationships between the structural parameters of $G$ and $\widehat{G}$, a more systematic investigation of this connection could further extend the scope of our analysis. 
Indeed, since $\xi(G \odot H)$ becomes a linear function of $\n(H)$ beyond the threshold $\min\{\alpha(\widehat{G}), \beta(\widehat{G})-\beta^*(G)+1\}$, with slope $\beta(\widehat{G})$, the study of this parameter may provide further insight into the structural behaviour of the corona product. 
In particular, it would be of interest to derive explicit expressions or sharp bounds for the vertex cover number $\beta(\widehat{G})$ in terms of well-known invariants of $G$, especially for graph families not considered in this work. Such results would further clarify the role of $\widehat{G}$ and potentially extend the applicability of our methods.

An open problem that remains to be investigated is whether the threshold $\min\{\alpha(\widehat{G}), \beta(\widehat{G})-\beta^*(G)+1\}$, beyond which $\xi(G \odot H)$ becomes a linear function of $\n(H)$, can be reduced. In fact, the examples of graphs $G$ we have presented (Figures~\ref{fig:FigFish} and~\ref{fig:NonLinearity}), for which this linearity fails, correspond to the cases $\n(H)=1$ and $\n(H)\le 2$, respectively. In neither case do these values reach the threshold established for linearity.

The fact that, for any fixed connected graph $G$, the equidistant dimension
of $G \odot H$ eventually becomes linear in $\n(H)$, with slope
$\beta(\widehat{G})$, leads us to ask when two non-isomorphic graphs $G_1$
and $G_2$ yield isomorphic empty bisector graphs. Identifying such conditions would explain when different base graphs induce identical asymptotic behaviour of the equidistant dimension under the corona product with arbitrary graphs $H$.

Finally, we leave open the problem of characterizing the graphs for which the
lower and upper bounds established in Theorems~\ref{improvedlbound}
and~\ref{UpperBound1}, respectively, are attained. Although we have identified
a sufficient condition under which the lower bound in
Theorem~\ref{improvedlbound} is achieved, a complete characterization of such
graphs is still missing. A full description of these extremal cases would
contribute to a deeper understanding of the tightness of our results and of the
structural features governing the equidistant dimension of corona products.


\begin{acknowledgement}

Sandor E. Tu{\~n}\'{o}n-Andr\'{e}s was supported by the Mart{\'i}-Franqu{\`e}s Predoctoral Research Grant Programme (2024PMF-PIPF-17) from Universitat Rovira i Virgili. This work was also supported by the project PROVTOPIA (PID2023-150098OB-I00) funded by MICIU/AEI/10.13039/501100011033 and FEDER (EU).

\end{acknowledgement}


\begin{thebibliography}{10}
\expandafter\ifx\csname url\endcsname\relax
  \def\url#1{\texttt{#1}}\fi
\expandafter\ifx\csname urlprefix\endcsname\relax\def\urlprefix{URL }\fi

\bibitem{BarikPatiSarma2007}
S.~Barik, S.~Pati, B.~K. Sarma, The spectrum of the corona of two graphs, SIAM
  Journal on Discrete Mathematics 21~(1) (2007) 47--56.
\newline\urlprefix\url{https://doi.org/10.1137/050624029}

\bibitem{Beerliova2006}
Z.~Beerliova, F.~Eberhard, T.~Erlebach, A.~Hall, M.~Hoffmann, M.~Mihal{\'a}k,
  L.~Ram, S.~S. Rao, Network discovery and verification, IEEE Journal on
  Selected Areas in Communications 24~(12) (2006) 2168--2181.
\newline\urlprefix\url{https://ieeexplore.ieee.org/document/4016133}

\bibitem{Gallai1959}
T.~Gallai, {\"{U}ber extreme Punkt- und Kantenmengen}, Annales Universitatis
  Scientarium Budapestinensis de Rolando E\"{o}tv\"{o}s Nominatae, Sectio
  Mathematica 2 (1959) 133--138.

\bibitem{GispertFernandez2024}
A.~Gispert-Fern{\'a}ndez, J.~A. Rodr{\'i}guez-Vel{\'a}zquez, The equidistant
  dimension of graphs: Np-completeness and the case of lexicographic product
  graphs, AIMS Mathematics 9~(6) (2024) 15325--15345.
\newline\urlprefix\url{https://www.aimspress.com/article/doi/10.3934/math.2024744}

\bibitem{GispertFernandez2025}
A.~Gispert-Fern{\'a}ndez, J.~A. Rodr{\'i}guez-Vel{\'a}zquez, I.~G. Yero,
  Equidistant dimension of cartesian product graphs (2025).
\newline\urlprefix\url{https://arxiv.org/abs/2512.07672}

\bibitem{Gonzalez2022}
A.~Gonz{\'a}lez, M.~C. Hernando, M.~Mora, The equidistant dimension of graphs,
  Bulletin of the Malaysian Mathematical Sciences Society 45 (2022)
1757--1775.
\newline\urlprefix\url{https://link.springer.com/article/10.1007/s40840-022-01295-z}

\bibitem{Harary1969}
F.~Harary, Graph theory, Addison-Wesley Publishing Co., Reading, Mass.-Menlo
  Park, Calif.-London, 1969.
\newline\urlprefix\url{http://books.google.es/books/about/Graph_Theory.html?id=9nOljWrLzAAC&redir_esc=y}

\bibitem{Harary1976}
F.~Harary, R.~A. Melter, On the metric dimension of a graph, Ars Combinatoria 2
  (1976) 191--195.
\newline\urlprefix\url{http://www.ams.org/mathscinet-getitem?mr=0457289}

\bibitem{Haynes1998a}
T.~W. Haynes, S.~T. Hedetniemi, P.~J. Slater, Domination in Graphs: Volume 2:
  Advanced Topics, Chapman \& Hall/CRC Pure and Applied Mathematics, Taylor \&
  Francis, 1998.

\bibitem{Imrich2000}
W.~Imrich, S.~Klav\v{z}ar, Product graphs, structure and recognition,
  Wiley-Interscience series in discrete mathematics and optimization, Wiley,
  2000.
\newline\urlprefix\url{http://books.google.es/books?id=EOnuAAAAMAAJ}

\bibitem{Khuller1996}
S.~Khuller, B.~Raghavachari, A.~Rosenfeld, Landmarks in graphs, Discrete
  Applied Mathematics 70~(3) (1996) 217--229.
\newline\urlprefix\url{http://www.sciencedirect.com/science/article/pii/0166218X95001062}

\bibitem{Kratica2024}
J.~Kratica, M.~{\v{C}}angalovi{\'c}, V.~Kova{\v{c}}evi{\'c}-Vuj{\v{c}}i{\'c},
  Equidistant dimension of johnson and kneser graphs, arXiv preprint (2025).
\newline\urlprefix\url{https://arxiv.org/abs/2406.17870}

\bibitem{Kratica2025}
J.~Kratica, A.~Savi{\'c}, Z.~Maksimovi{\'c}, M.~Bogdanovi{\'c}, On the
  equidistant dimension of hamming graphs, Matemati{\v{c}}ki Vesnik (2025)
  1--7.
\newline\urlprefix\url{https://doi.org/10.57016/MV-ZooS4049}

\bibitem{Savic2024}
A.~Savi\'c, Z.~Maksimovi\'c, M.~Bogdanovi\'c, J.~Kratica, The equidistant
  dimension of some graphs of convex polytopes, arXiv preprint (2024).
\newline\urlprefix\url{https://arxiv.org/abs/2407.15307}

\bibitem{Slater1975}
P.~J. Slater, Leaves of trees, Congressus Numerantium 14 (1975) 549--559,
  introduces the concept of locating sets (resolving sets) in trees, considered
  an early contribution to metric dimension in graphs.

\bibitem{Trujillo-Rasua2016}
R.~Trujillo-Rasua, I.~G. Yero, $k$-metric antidimension: A privacy measure for
  social graphs, Information Sciences 328 (2016) 403--417.
\newline\urlprefix\url{http://www.sciencedirect.com/science/article/pii/S0020025515006477}

\bibitem{Warshall1962}
S.~Warshall, A theorem on boolean matrices, Journal of the ACM 9~(1) (1962)
  11--12.
\newline\urlprefix\url{http://doi.acm.org/10.1145/321105.321107}

\bibitem{West2001}
D.~B. West, et~al., Introduction to graph theory, vol.~2, Prentice hall Upper
  Saddle River, 2001.

\bibitem{Yero2011}
I.~G. Yero, D.~Kuziak, J.~A. Rodr\'{\i}guez-Vel\'{a}zquez, On the metric
  dimension of corona product graphs, Computers \& Mathematics with
  Applications 61~(9) (2011) 2793--2798.
\newline\urlprefix\url{http://www.sciencedirect.com/science/article/pii/S0898122111002094}

\end{thebibliography}
\end{document}